\documentclass[a4paper,
               11pt,
               pdftex,
               normalheadings,
               headsepline,
               footsepline,
               onecolumn,
               headinclude,
               footinclude,
               DIV14,
               abstracton]
{scrartcl}

\usepackage
{
    graphicx,
    amssymb,
    amsmath,
    amsthm,
    xcolor,
    dsfont,
    algpseudocode,
    authblk,
}

\usepackage[ngerman, english]{babel}

\usepackage[left=30mm,right=30mm,top=35mm,bottom=40mm]{geometry}

\usepackage[colorlinks,
            pdffitwindow=false,
            plainpages=false,
            pdfpagelabels=true,
            pdfpagemode=UseOutlines,
            pdfpagelayout=SinglePage,
            bookmarks=false,
            colorlinks=true,
            hyperfootnotes=false,
            linkcolor=blue,
            citecolor=green!50!black]
{hyperref}

\usepackage{enumerate}
\usepackage{algorithm}

\usepackage{tabularx}

\usepackage{hyphenat}
\usepackage{makeidx}         
\usepackage{graphicx}        
\usepackage{multicol}        
\usepackage[bottom]{footmisc}

\newcommand{\R}{\mathbb{R}}
\newcommand{\N}{\mathbb{N}}

\newcommand{\Fcal}{\mathcal{F}}
\newcommand{\Kop}{\mathcal{U}}
\newcommand{\SetState}{\mathcal{Y}}
\newcommand{\SetControl}{\mathfrak{U}}

\newcommand\xqed[1]{\leavevmode\unskip\penalty9999 \hbox{}\nobreak\hfill \quad\hbox{#1}}
\newcommand{\exampleSymbol}{\xqed{$\triangle$}}

\newcolumntype{Y}{>{\centering\arraybackslash}X}

\newtheorem{theorem}{Theorem}[section]
\newtheorem{corollary}[theorem]{Corollary}
\newtheorem{lemma}[theorem]{Lemma}

\newtheorem{remark}[theorem]{Remark}
\newtheorem{example}[theorem]{Example}
\newtheorem{assumption}[theorem]{Assumption}

\makeatletter
\renewcommand*\env@matrix[1][*\c@MaxMatrixCols c]{%
  \hskip -\arraycolsep
  \let\@ifnextchar\new@ifnextchar
  \array{#1}}
\makeatother

\newcolumntype{Y}{>{\centering\arraybackslash}X}

\let\OLDthebibliography\thebibliography
\renewcommand\thebibliography[1]{
	\OLDthebibliography{#1}
	\setlength{\parskip}{0pt}
	\setlength{\itemsep}{0pt plus 0.3ex}
}

\renewcommand\vec{\mathbf}

\date{}

\begin{document}

\title{Feedback control of nonlinear PDEs using data-efficient reduced order models based on the Koopman operator}
\author[1]{Sebastian Peitz}
\author[2]{Stefan Klus}
\affil[1]{\normalsize Department of Mathematics, Paderborn University, Germany}
\affil[2]{\normalsize Department of Mathematics and Computer Science, Freie Universit{\"a}t Berlin, Germany}

\maketitle

\begin{abstract}
In the development of model predictive controllers for PDE-constrained problems, the use of reduced order models is essential to enable real-time applicability. Besides local linearization approaches, Proper Orthogonal Decomposition (POD) has been most widely used in the past in order to derive such models. Due to the huge advances concerning both theory as well as the numerical approximation, a very promising alternative based on the Koopman operator has recently emerged. In this chapter, we present two control strategies for model predictive control of nonlinear PDEs using data-efficient approximations of the Koopman operator. In the first one, the dynamic control system is replaced by a small number of autonomous systems with different yet constant inputs. The control problem is consequently transformed into a switching problem. In the second approach, a bilinear surrogate model, is obtained via linear interpolation between two of these autonomous systems. Using a recent convergence result for Extended Dynamic Mode Decomposition (EDMD), convergence to the true optimum can be proved. We study the properties of these two strategies with respect to solution quality, data requirements, and complexity of the resulting optimization problem using the 1D Burgers Equation and the 2D Navier--Stokes Equations as examples. Finally, an extension for online adaptivity is presented.
\end{abstract}

\section{Introduction}
\label{sec:Introduction}

The control of systems governed by nonlinear partial differential equations (PDEs) is very challenging. Since classical control strategies are difficult to develop in this context, advanced techniques such as \emph{Model Predictive Control (MPC)} \cite{GP17} are very popular. In MPC, an open-loop optimal control problem is repeatedly solved online over a finite-time horizon using a model of the system dynamics, which then results in a closed-loop controller. The PDE-constrained optimal control problems we are interested in are of the following form: 
\begin{equation}\label{eq:OCP} \tag{OCP}
	\begin{aligned}
		\min_{u \in \SetControl} J(\vec{y}) &= \min_{u \in \SetControl} \int_{t_0}^{t_e}L(\vec{y}(\vec{x},t)) \ dt \\
		\mbox{s.t.} \quad \dot{\vec{y}}(\vec{x},t) &= G(\vec{y}(\vec{x},t),u(t)), \\
		\vec{y}(0) &= \vec{y}^0.
	\end{aligned}
\end{equation}
where $\vec{y} \in \SetState$ is the system state (depending on the $d$-dimensional space coordinate $\vec{x} \in \Omega \subseteq \R^d$ and the time $t \in \R^{\geq 0}$), $u \in \SetControl$ is the control function, and the partial differential operator $G: \SetState \times \SetControl \rightarrow \SetState$ describes the system dynamics. For ease of notation, the term in the objective function $L: \SetState \rightarrow \R$ does not depend on $u$ explicitly.

The downside of MPC is that the open-loop control problem~\eqref{eq:OCP} has to be solved in a short amount of time, which is generally not possible for PDE-constrained problems when using a standard discretization approach such as finite elements or finite volumes. 
A remedy to this issue is \emph{reduced order modeling (ROM)}, where the high-fidelity model is replaced by a low-dimensional surrogate model, see \cite{LMQR14,BGW15} for overviews. In the nonlinear case, the method of \emph{Proper Orthogonal Decomposition (POD)} \cite{Sir87} has been successfully applied in a large variety of problems concerning both simulation and control. In the latter case, there exist different approaches to ensure convergence towards the true optimum. The two most popular are to derive error bounds based on the singular values associated with the POD modes \cite{KV99,Row05,HV05,TV09} and to adapt classical \emph{trust-region} approaches to surrogate modeling \cite{Fah00,BC08,QGVW16}.

A much more recent approach to develop a ROM is via the linear but infinite-dimensional \emph{Koopman operator} \cite{Koo31}, which describes the dynamics of observables. In the past decade, significant advances were obtained concerning theoretical aspects of the Koopman operator \cite{MB04,Mez05,BMM12,Mez13} as well as its numerical approximation via \emph{Dynamic Mode Decomposition (DMD)} \cite{Sch10,RMB+09,TRL+14} or \emph{Extended Dynamic Mode Decomposition (EDMD)} \cite{WKR15,KGPS18,KKS16}. An advantage over POD is that this approach can also be applied in situations where the underlying system dynamics is unknown. 

The above-mentioned advancements have led to several approaches for including the Koopman operator in control frameworks, see, e.g., \cite{PBK15,PBK18,BBPK16,KM16,KKB17,AKM18}. In many of these approaches, the Koopman operator is approximated for an augmented state (consisting of the actual state and the control) in order to deal with the non-autonomous control system. For this reason, a large amount of data is necessary to cover a sufficient range of the dynamics. Alternative approaches have recently been presented by the authors in \cite{PK17} and \cite{Pei18}. 
Since the Koopman operator is only applicable to autonomous systems in its original formulation, we take the following two steps:
\begin{enumerate}[i)]
	\item replace the control system $G$ by a finite number of autonomous systems $G_{u^j}$ with constant input $u^j$,
	\item construct reduced order models for low-dimensional observations (instead of the entire state) of $G_{u^j}$ using the corresponding Koopman operator $\Kop_{u^j}$.
\end{enumerate}
In a third step, the PDE constraint in Problem~\eqref{eq:OCP} is replaced by the reduced model. We perform this step in two different ways (cf.~\cite{PK17} and \cite{Pei18}, respectively, for details):
\begin{enumerate}[{iii-}a)]
	\item transform the optimization problem into a switching problem (which of the autonomous systems has to be applied in each time step?),
	\item construct a bilinear surrogate model via linear interpolation between two Koopman operators $\Kop_{u^0}$ and $\Kop_{u^1}$.
\end{enumerate}
In this way, convergence towards the true optimum can be shown by utilizing a recent convergence result for EDMD \cite{KM17}.

Since reduced order modeling approaches using the Koopman operator are relatively new, people have much less experience in this direction compared to more established methods such as POD. The purpose of this chapter is therefore to study the two approaches described above regarding the numerical performance. We address both the quality of the solution compared to the PDE-constrained problem as well as the influence of the training data. Furthermore, the effect of introducing the switching problem transformation is studied.

The remainder of this chapter is structured as follows. In Section~\ref{sec:Koopman}, the notation for the Koopman operator is introduced and the reduced order modeling approach for low-dimensional observations is presented. In Section~\ref{sec:STMPC}, we give a short introduction to model predictive control which we use to realize feedback behavior. We then introduce the two reduced order modeling strategies in Section~\ref{sec:KROM} before studying numerical properties and the control performance in Section~\ref{sec:InfluenceDataBasis}. Finally, we use the concept from \cite{HWR14} to obtain online updates for the reduced models in Section~\ref{sec:OnlineUpdates} before drawing a conclusion in Section~\ref{sec:Conclusion}.

\section{Reduced order modeling using the Koopman operator}
\label{sec:Koopman}
Let $ \Phi : \SetState \to \SetState $ be a discrete deterministic dynamical system defined on the state space $\SetState$ and let $ f \colon \SetState \rightarrow \R $ be a real-valued observable of the system.\footnote{The state space $\SetState$ can either be a finite-dimensional (i.e., $\SetState \subseteq \R^d$) or an infinite-dimensional space \cite{KM17}. An extension to vector-valued observables can be obtained in a straightforward manner \cite{BMM12}.} Then the Koopman operator $ \Kop \colon \Fcal \to \Fcal $ with $\Fcal = L^{\infty}(\SetState)$, see~\cite{LaMa94,BMM12,Mez13,WKR15}, which describes the evolution of the observable $ f $, is defined by
\begin{equation*}
	(\Kop f)(\vec{y}) = f(\Phi(\vec{y})).
\end{equation*}
The Koopman operator is linear but infinite-dimensional. Its adjoint, the Perron--Frobenius operator, describes the evolution of densities. The definition of the Koopman operator can be naturally extended to continuous-time dynamical systems as described in \cite{LaMa94,BMM12}. Given an autonomous system of the form 
\begin{equation*}
	\dot{\vec{y}}(t) = G(\vec{y}(\vec{x},t)),
\end{equation*}
the \emph{Koopman semigroup} of operators $ \{ \Kop^t \} $ is defined as
\begin{equation*}
	(\Kop^t f)(\vec{y}) = f(\Phi^t(\vec{y})),
\end{equation*}
where $ \Phi^t $ is the flow map associated with $ G $. In what follows, we will mainly consider discrete dynamical systems, given by the discretization of ODEs or PDEs. That is, $ \Phi = \Phi^h $ for a fixed time step $ h $.

One method to compute a numerical approximation of the Koopman operator from data is EDMD \cite{WKR15,KKS16}. The following brief description is based on the review paper \cite{KNKWKSN18}. EDMD is a generalization of DMD \cite{Sch10,TRL+14} and can be used to compute a finite-dimensional approximation of the Koopman operator, its eigenvalues, eigenfunctions, and modes. In contrast to DMD, EDMD allows arbitrary basis functions -- which could be, for instance, monomials, Hermite polynomials, or trigonometric functions -- for the approximation of the dynamics. We do not observe the full (potentially infinite-dimensional) state of the system, but consider only a finite number of measurements, given by $ \vec{z} = f(\vec{y}) \in \R^q $. The special case $f = \mathsf{Id}$ is known as the \emph{full state observable}. For a given set of basis functions $ \{ \psi_{1},\,\psi_{2},\,\dots,\,\psi_{k} \} $, we then define a vector-valued function $ \psi \colon \R^q \to \R^k $ by
\begin{equation*}
	\psi(\vec{z}) =
	\begin{bmatrix}
		\psi_{1}(\vec{z}) & \psi_{2}(\vec{z}) & \dots & \psi_{k}(\vec{z})
	\end{bmatrix}^{\top}.
\end{equation*}
If $ \psi(\vec{z}) = \vec{z} $, we obtain DMD as a special case of EDMD. We assume that we have either measurement or simulation data, written in matrix form as
\begin{equation*}
	\vec{Z} =
	\begin{bmatrix}
		\vec{z}_1 & \vec{z}_2 & \cdots & \vec{z}_m
	\end{bmatrix}
	\quad\text{and}\quad
	\widetilde{\vec{Z}} =
	\begin{bmatrix}
		\widetilde{\vec{z}}_1 & \widetilde{\vec{z}}_2 & \cdots & \widetilde{\vec{z}}_m
	\end{bmatrix},
\end{equation*}
where $ \widetilde{\vec{z}}_i = f(\Phi(\vec{y}_i)) $. The data could either be obtained via many short simulations or experiments with different initial conditions or one long-term trajectory or measurement. If the data is extracted from one long trajectory, then $ \widetilde{\vec{z}}_i = \vec{z}_{i+1} $. The data matrices are embedded into the typically higher-dimensional feature space by
\begin{align*}
	\Psi_{\vec{Z}} =
	\begin{bmatrix} \psi(\vec{z}_{1}) & \psi(\vec{z}_{2}) & \dots & \psi(\vec{z}_{m}) \end{bmatrix}
	\quad \text{and} \quad
	\Psi_{\widetilde{\vec{Z}}} =
	\begin{bmatrix} \psi(\widetilde{\vec{z}}_{1}) & \psi(\widetilde{\vec{z}}_{2}) & \dots & \psi(\widetilde{\vec{z}}_{m}) \end{bmatrix}.
\end{align*}
With these data matrices, we then compute the matrix $ \vec{U} \in \R^{k \times k} $ defined by
\begin{equation*}
	\vec{U}^{\top} = \Psi_{\widetilde{\vec{Z}}} \Psi_{\vec{Z}}^+ = \big( \Psi_{\widetilde{\vec{Z}}} \Psi_{\vec{Z}}^{\top} \big) \big(\Psi_{\vec{Z}} \Psi_{\vec{Z}}^{\top}\big)^+,
\end{equation*}
where $^+$ denotes the pseudoinverse. 
The matrix $ \vec{U} $ can be viewed as a finite-dimensional approximation of the Koopman operator. 

\begin{figure}[b]
	\centering
	\includegraphics[width=0.55\textwidth]{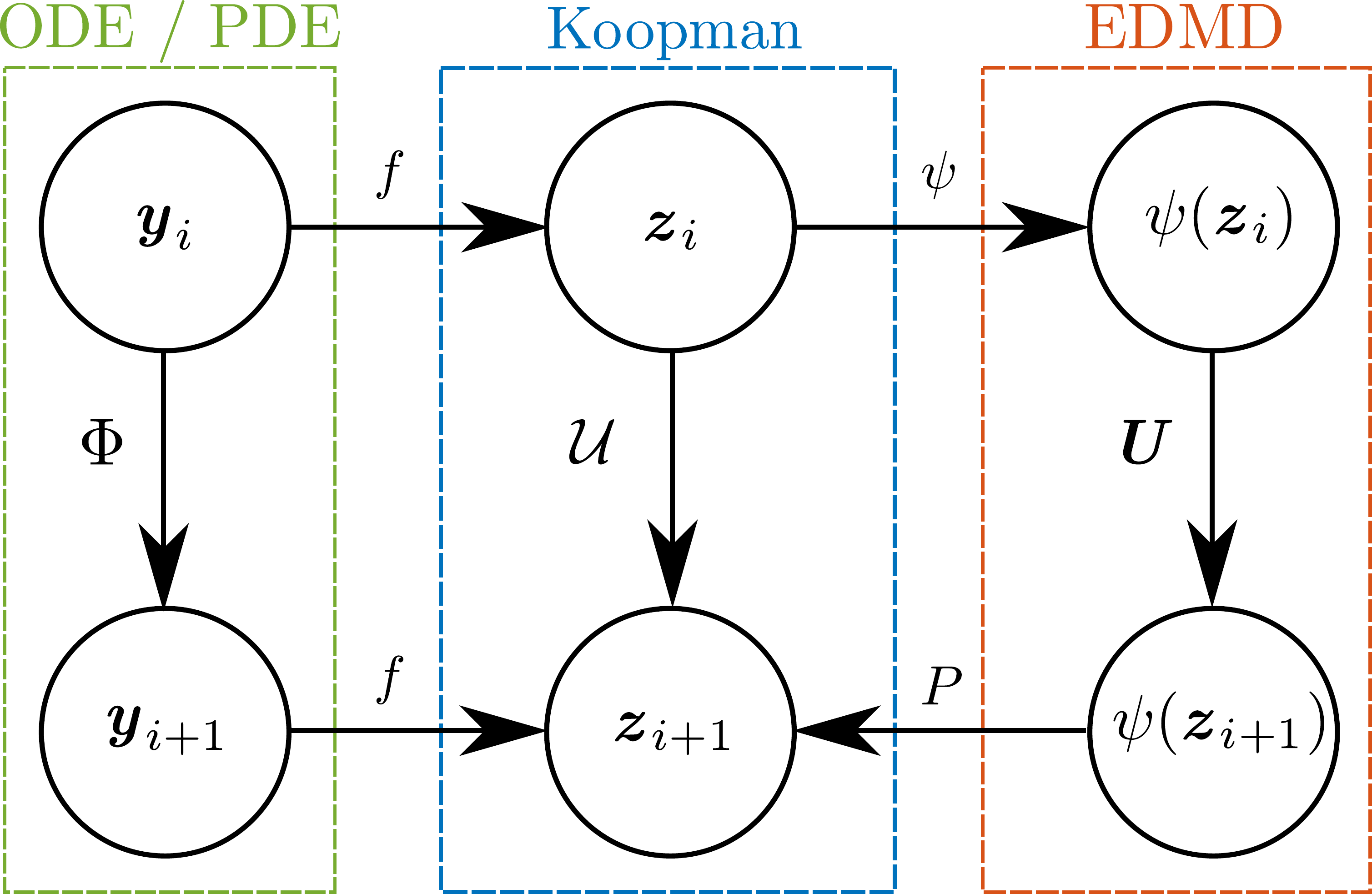}
	\caption{Relation between the system dynamics $\Phi$, the corresponding Koopman operator $\Kop$ and its finite-dimensional representation $\vec{U}$ computed via EDMD.}
	\label{fig:Koopman_EDMD}
\end{figure}

Convergence of EDMD in the infinite-data limit for a fixed set of basis functions was first analyzed in \cite{WKR15, KKS16}. Convergence towards the Koopman operator for the case that also the number of basis functions goes to infinity has recently been proven in~\cite{KM17}. Under some assumptions such as 
independent drawing of data points with respect to some given probability measure $ \mu $
and boundedness of the Koopman operator, the EDMD approximation $\vec{U}$ converges to the Koopman operator for $k \to \infty$ and $m\to\infty$ provided that $(\psi_i)_{i=1}^{\infty}$ is an orthonormal basis of $\Fcal$. For the convergence results below, we assume that these conditions are satisfied.

The decomposition of the Koopman operator into modes, eigenvalues and eigenfunctions is commonly used to analyze the system dynamics as well as predict the future state. In the situation we are presenting here, we can pursue an even simpler approach and obtain the update for the observable $\vec{z}$ directly using $\vec{U}$ which yields the Koopman operator based reduced order model:
\begin{equation}\label{eq:KROM} \tag{K-ROM}
	\psi(\vec{z}_{i+1}) = \vec{U}^{\top} \psi(\vec{z}_{i}), \quad i = 0,1,\ldots
\end{equation}
This approach is visualized in Figure~\ref{fig:Koopman_EDMD}, where we see that
\begin{align*}
	f \circ \Phi = \Kop f \approx P \circ \vec{U}^\top \circ \psi \circ f.
\end{align*}
If we let the number of data points as well as the number of basis functions go to infinity, then the EDMD approximation converges to the Koopman operator as discussed above.

Table~\ref{tab:Sampling} shows the efficiency of the K-ROM for the two example which we will consider throughout this chapter, the 1D Burgers equation and the 2D Navier--Stokes equations. Here, we have chosen a monomial basis for the dictionary $\Psi$. The maximum order of the monomials and the number of observations yield the K-ROM dimension. We see that the dimension of the PDE-constrained problem (obtained by a finite-difference (FD) approximation) is reduced by a factor of $1.4$ in the Burgers example. Although this does not seem like a large reduction, we obtain a speed-up of approximately $100$ since \eqref{eq:KROM} is linear on the one hand and we can choose larger step sizes in the reduced model on the other hand. For the Navier--Stokes example, we additionally have a major reduction of the dimension of the state compared to the finite volume (FV) discretization by a factor of almost $500$. This way, a speed-up of $75,000$ is achieved.

\begin{table}[h!]
	\centering
	\caption{Numerical setup and efficiency analysis for two example problems.}
	\begin{tabular}{p{6cm}p{2.5cm}p{2.5cm}}
		\hline\noalign{\smallskip}
		Problem & 1D Burgers & 2D NSE \\ 
		\noalign{\smallskip}\hline\noalign{\smallskip}
		Order of monomials & 3 & 2 \\ 
		dim($\vec{z}$) & 4 & 8 \\ 
		dim($\psi(\vec{z})$) = dim($\vec{U}_{u^j}$) & 35 & 45 \\ 
		dim($\vec{y}$) & 49 (FD) & 22,000 (FV) \\ 
		dim($\vec{y}$) / dim($\psi(\vec{z})$) & 1.4 & 488.9 \\ 
		Speed-up \eqref{eq:KROM} vs.~PDE & $\approx 100$ & $\approx 7.5 \cdot 10^4$ \\ 
		\noalign{\smallskip}\hline\noalign{\smallskip}
	\end{tabular}
	\label{tab:Sampling}
\end{table}

\section{Model Predictive Control}
\label{sec:STMPC}
Due to uncertainties and noise, open-loop control strategies are insufficient for applications. To overcome this issue, in MPC \cite{GP17}, 
the infinite-horizon open-loop problem is split into finite-horizon problems with a \emph{prediction horizon} of length $p$:
\begin{equation}\label{eq:MPC} \tag{MPC}
	\begin{aligned}
		&\min_{\vec{u} \in \R^p} \sum_{i=s}^{s+p-1} L(\vec{y}_i) \\
		\text{s.t.}\quad \vec{y}_{i+1} &= \Phi(\vec{y}_i, \vec{u}_{i-s+1}), \quad i = s,\ldots,s + p - 1, \\
		\vec{y}_s &= \vec{y}^s,
	\end{aligned}
\end{equation}
where $\vec{y}^s$ is the initial condition obtained (or approximated) from sensor data. The first part of this solution is then applied to the real system while the optimization is repeated with the prediction horizon moving forward by one sample time. In Problem~\eqref{eq:MPC}, the system dynamics are of discrete form since the control input is constant over each sample time interval. When dealing with continuous-time systems such as Problem~\eqref{eq:OCP}, this formulation can be regarded as the flow map $\Phi^h$ (cf.~Section~\ref{sec:Koopman}) of the continuous dynamics with time step $h$.

A consequence of the MPC method is that Problem~\eqref{eq:MPC} has to be solved online, i.e., within the time step $h$. Since this is in general impossible for PDE-constrained problems, we will present two approaches to replace the PDE constraint in \eqref{eq:MPC} by a Koopman operator based reduced order model in the next section.

\section{A data-efficient method to construct Koopman operator based surrogate models}
\label{sec:KROM}
As already outlined in the introduction, we will construct K-ROMs with significant speed-up factors. In order to do this in a data-efficient way, we perform steps i) and ii) mentioned there. 
The first step ensures that our data requirements are moderate since we only have to collect data for a finite set of inputs. The second steps allows us to derive linear models with a low dimension. Since the approach is entirely data-based and hence equation-free, we can use any observation and in particular those which are relevant for the control task at hand.

In a third step, we have to transform the optimal control problem accordingly which we will do in two different ways. In Section~\ref{subsec:KROM_Switched}, the optimization problem is transformed into a switching problem and in Section~\ref{subsec:KROM_Bilinear}, a bilinear model is constructed via linear interpolation.

\subsection{Transformation to switched systems}
\label{subsec:KROM_Switched}
A large variety of technical systems is controlled via switching between different inputs. Examples are valves in chemical reactors which are either open or closed or the switching of gears in electrical drives. These so-called \emph{switched systems} can be regarded as a special case of hybrid systems which possess both continuous and discrete-time control inputs (cf.~\cite{ZA15} for a survey). 
We here make use of the concept of switched systems in order to reduce the data that is required for the training process of the K-ROM. To this end, we replace the right-hand side of the dynamical control system by a finite set of autonomous systems which is achieved by fixing the input to $n_c$ different constant values $\{u^0,\ldots,u^{n_c-1}\}$ in \eqref{eq:OCP}. This yields $n_c$ different differential operators, $G_{u^0}, \ldots, G_{u^{n_c-1}}$, and the respective flow maps $\Phi_{u^0}, \ldots, \Phi_{u^{n_c-1}}$.

A consequence of this approach is that the optimization problem is transformed into a combinatorial problem where we have to select the optimal right-hand side in each time step. Since combinatorial problems are often more challenging to solve, we can fix the sequence in which the different autonomous systems are used, i.e., we switch from $G_{u^0}$ to $G_{u^1}$, from $G_{u^1}$ to $G_{u^2}$ and so on. Having reached the final system, we go back from $G_{u^{n_c-1}}$ to $G_{u^0}$. This way, the optimization variable again becomes real-valued, as we now have to compute the time instants for the switches $\vec{\tau} \in \R^{p+2}$ (with $\vec{\tau}_0 = t_0$ and $\vec{\tau}_{p+1} = t_e$):
\begin{equation}\label{eq:switched_dynamics}
	\begin{aligned}
		\dot{\vec{y}}(t) &= G_{u^j}(\vec{y}(\vec{x},t)) \quad \text{for}~t\in [\vec{\tau}_{l-1}, \vec{\tau}_l), \\
		\vec{y}(0) &= \vec{y}^0, \\
		j &= l~\mbox{mod}~n_c.
	\end{aligned}
\end{equation}
Using \eqref{eq:switched_dynamics}, we can reformulate the open-loop problem \eqref{eq:OCP} in terms of the switching instants
\begin{equation}\label{eq:STO} \tag{$\mbox{OCP}^{s}$} 
	\begin{aligned}
		\min_{\vec{\tau} \in \R^{p+2}} J(\vec{y}) &= \min_{\vec{\tau} \in \R^{p+2}} \int_{t_0}^{t_e}L(\vec{y}(\vec{x},t)) \ dt\\
		\mbox{s.t.} \qquad \qquad &\eqref{eq:switched_dynamics},\\
		\vec{y}(0) &= \vec{y}^0.
	\end{aligned}
\end{equation}
Different methods exist for efficiently computing solutions to \eqref{eq:STO}, see, e.g., \cite{EWD03,EWA06,SOBG16} for continuous-time or \cite{FMO13} for discrete-time problems. A second-order method has recently been proposed in \cite{SOBG17}. The following example illustrates the consequences of introducing such a switching control.

\begin{example}\label{ex:VdP}
Assume we want to control the behavior of the Van der Pol oscillator given by
\begin{equation*}
    \begin{aligned}
        \dot{\vec{y}}(t)
            &= G(\vec{y}(t),\,u(t))
            = \begin{pmatrix}
                \vec{y}_2(t) \\
                \big(1 - \vec{y}_1(t)^2\big) \, \vec{y}_2(t) - \vec{y}_1(t) + u(t)
              \end{pmatrix},\\
        \vec{y}(0) &= \vec{y}^0.
    \end{aligned}
\end{equation*}
By restricting the input $ u $ to $n_c$ values, we can transform the control system into $n_c$ autonomous systems of the form
\begin{equation*}
    \begin{aligned}
        \dot{\vec{y}}(t)
            &= G_{u^j}(\vec{y}(t))
            = \begin{pmatrix}
                \vec{y}_2(t) \\
                \big(1 - \vec{y}_1(t)^2\big) \, \vec{y}_2(t) - \vec{y}_1(t)
              \end{pmatrix}
              +
              \begin{pmatrix}
                0 \\
                u^j
              \end{pmatrix},\\
           \vec{y}(0) &= \vec{y}^0.
    \end{aligned}
\end{equation*}
The system dynamics for different input functions $ u $ are shown in Figure~\ref{fig:VdP}.

\begin{figure}[tbh]
    \centering
    \begin{minipage}{0.45\textwidth}
        \centering
        (a) \\ \includegraphics[width=\textwidth]{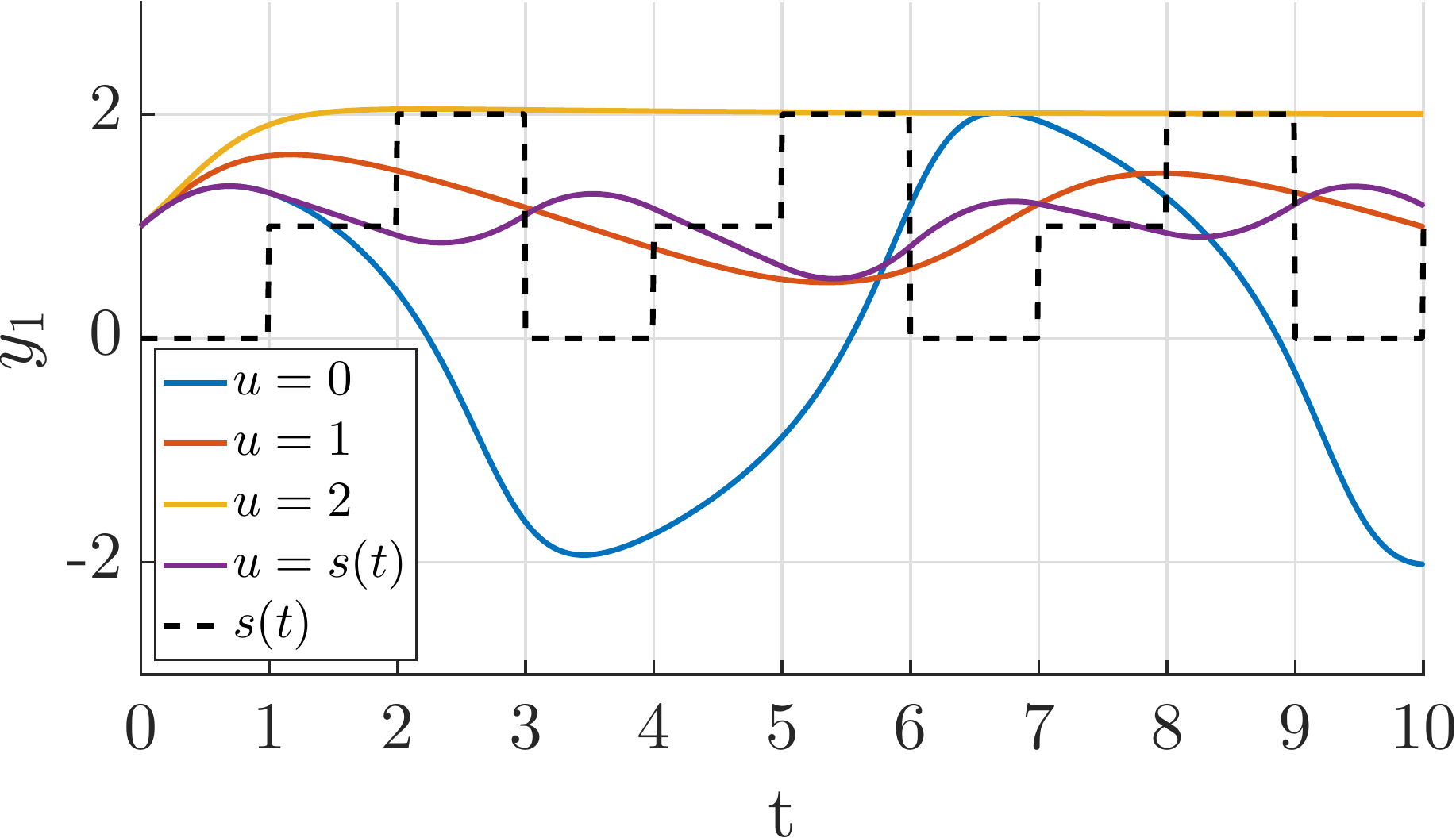}
    \end{minipage}\hfil
    \begin{minipage}{0.45\textwidth}
        \centering
        (b) \\ \includegraphics[width=\textwidth]{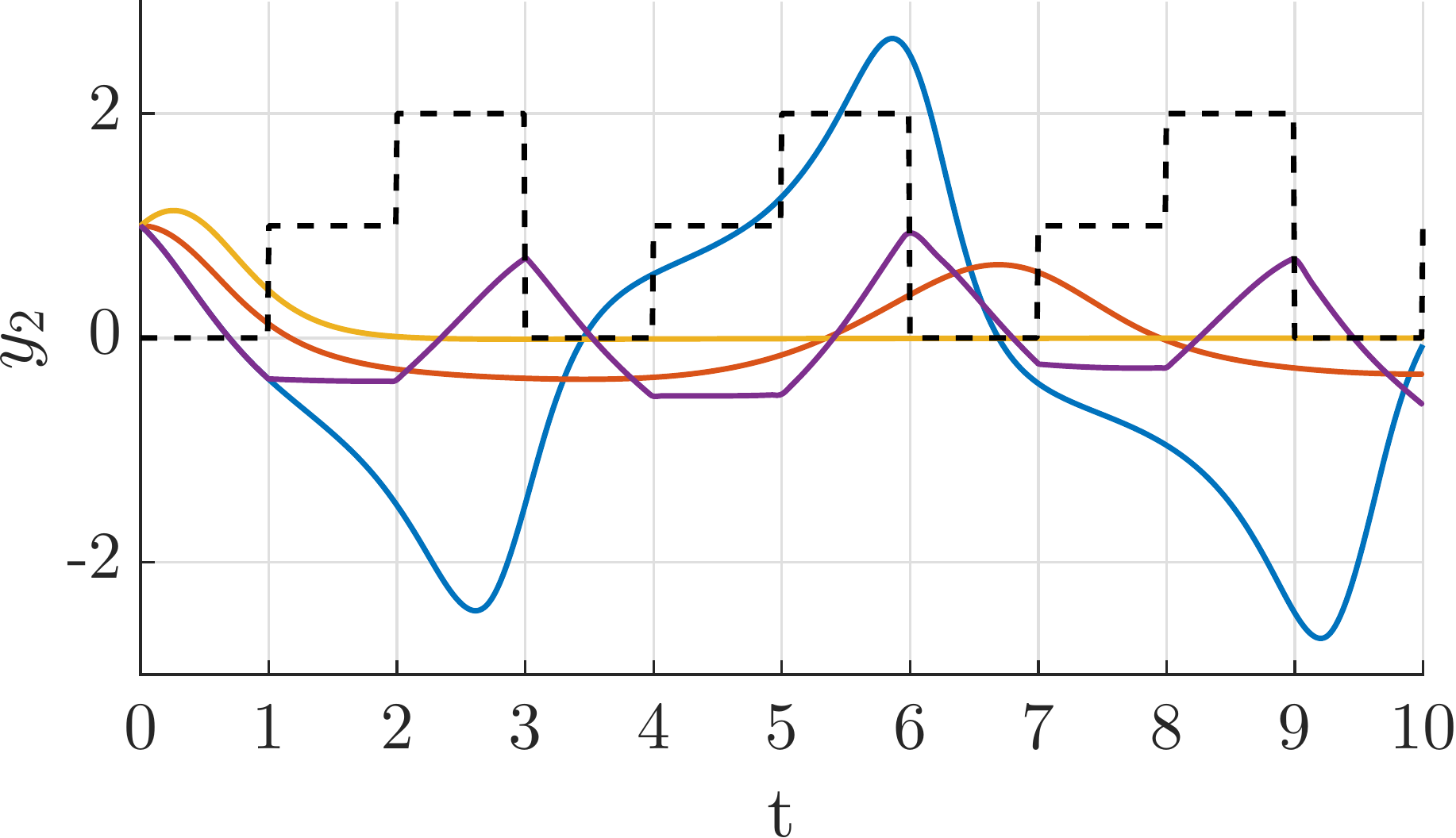}
    \end{minipage}
    \caption{The trajectories of three different autonomous systems with starting point $ \vec{y}^0 = (1, 1)^\top $ and the trajectory of the switched system with switching according to the dashed line.}
    \label{fig:VdP}
\end{figure}
\exampleSymbol
\end{example}

If we consider the discrete-time formulation \eqref{eq:MPC} for the closed-loop controller, the transformation is achieved in a very similar fashion
\begin{equation}\label{eq:MPC_STO} \tag{$\mbox{MPC}^{s}$}
	\begin{aligned} 
		&\min_{\vec{\tau} \in \{u^0,\ldots,u^{n_c-1}\}^p} \sum_{i=s}^{s+p-1} L(\vec{y}_i) \\
		\mbox{s.t.}\quad \vec{y}_{i+1} &= \Phi_{\vec{\tau}_{i-s+1}}(\vec{y}_i), \quad i = s,\ldots,s + p - 1, \\
		\vec{y}_{s} &= \vec{y}^s.
	\end{aligned}
\end{equation}
Since the input is constant over the sample time interval, we here obtain a combinatorial problem without a workaround as in the continuous-time case.
Each entry of $\vec{\tau}$ now describes which flow map $\Phi_{\vec{\tau}_i}$ to apply in the $i^{\mathsf{th}}$ step.

A popular approach to solve discrete-time optimal control problems of such form is via dynamic programming \cite{BD15}. However, this is only advisable if we are interested in larger prediction horizons $p$. For low values of $p$ (say, $3$), the most efficient way is indeed to evaluate all ${n_c}^p$ values of $\vec{\tau}$.

\subsubsection{Example: The 1D Burgers equation}
\label{subsubsec:ExampleBurgers}

We now study the difference between a continuous control input and a switched system approximation using the 1D Burgers Equation with periodic boundary conditions and a distributed control:
\begin{equation} \label{eq:BurgersContinuous}
	\begin{aligned}
		\dot{y}(x,t) - \nu \Delta y(x,t) + y(x,t) \nabla y(x,t) &= u(t) \chi(x),\\
		y(x,0) &= y^0(x).
	\end{aligned}
\end{equation}
The viscosity is set to $\nu = 0.01$ and the distributed control is realized by a time dependent scalar input $u$ and a shape function $\chi$ which is shown in Figure~\ref{fig:Burgers}~(a).
\begin{figure}[ht!]
	\centering
	\parbox[b]{0.45\textwidth}{\centering (a) \\ \includegraphics[width=.38\textwidth]{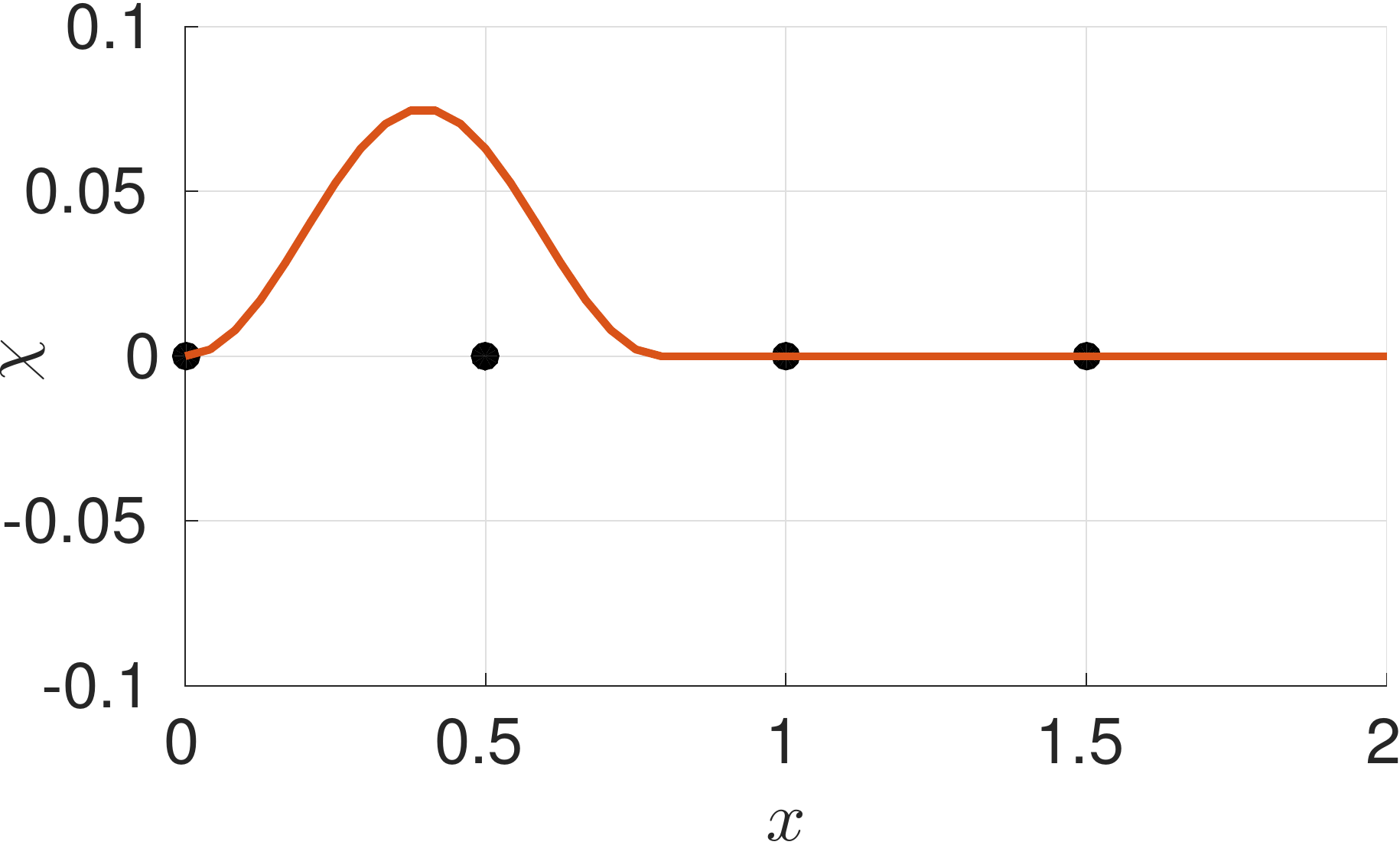}} \hfil
	\parbox[b]{0.45\textwidth}{\centering (b) \\ \includegraphics[width=.4\textwidth]{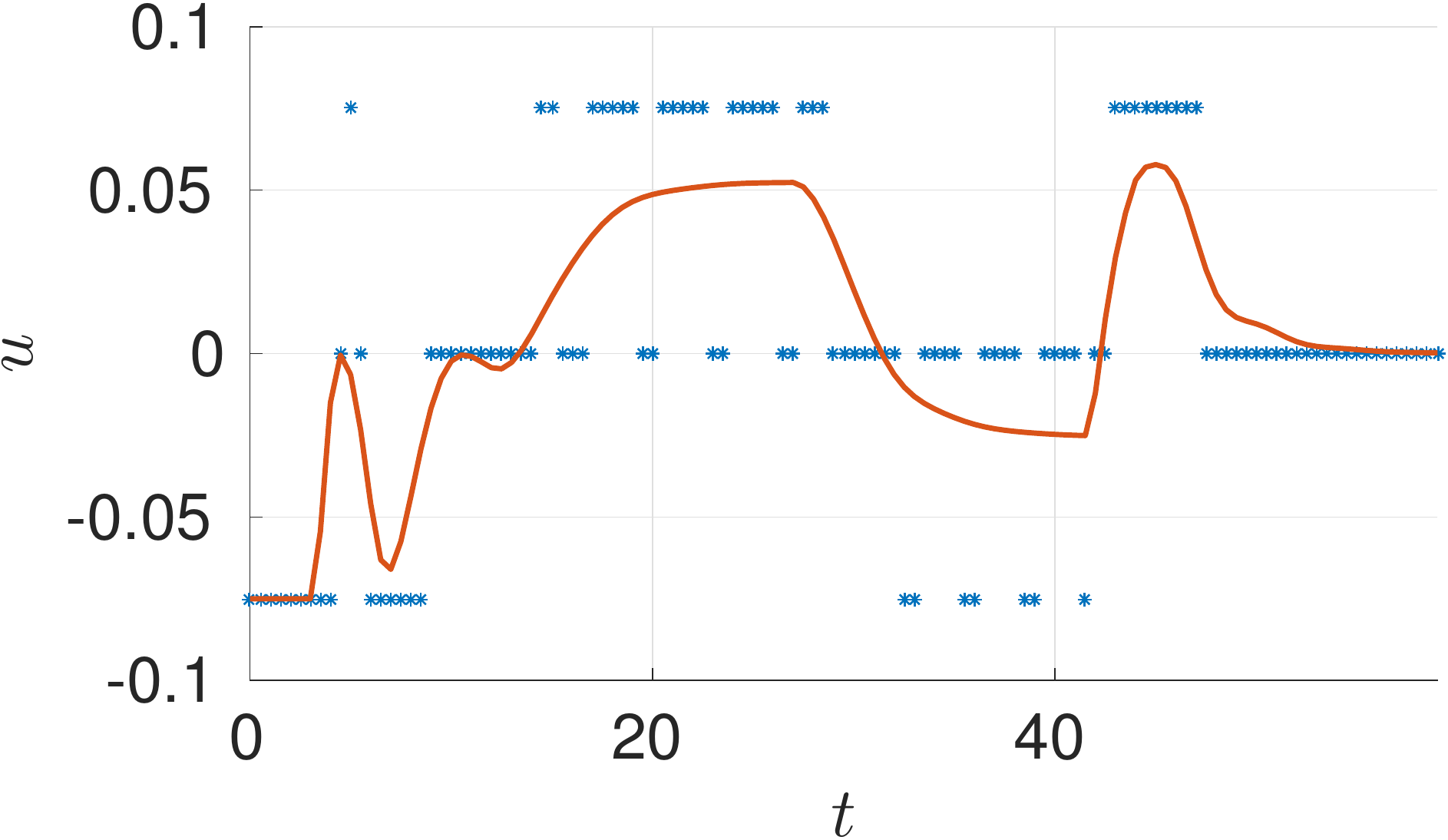}} \\[1ex]
	\parbox[b]{0.45\textwidth}{\centering (c) \\ \includegraphics[width=.4\textwidth]{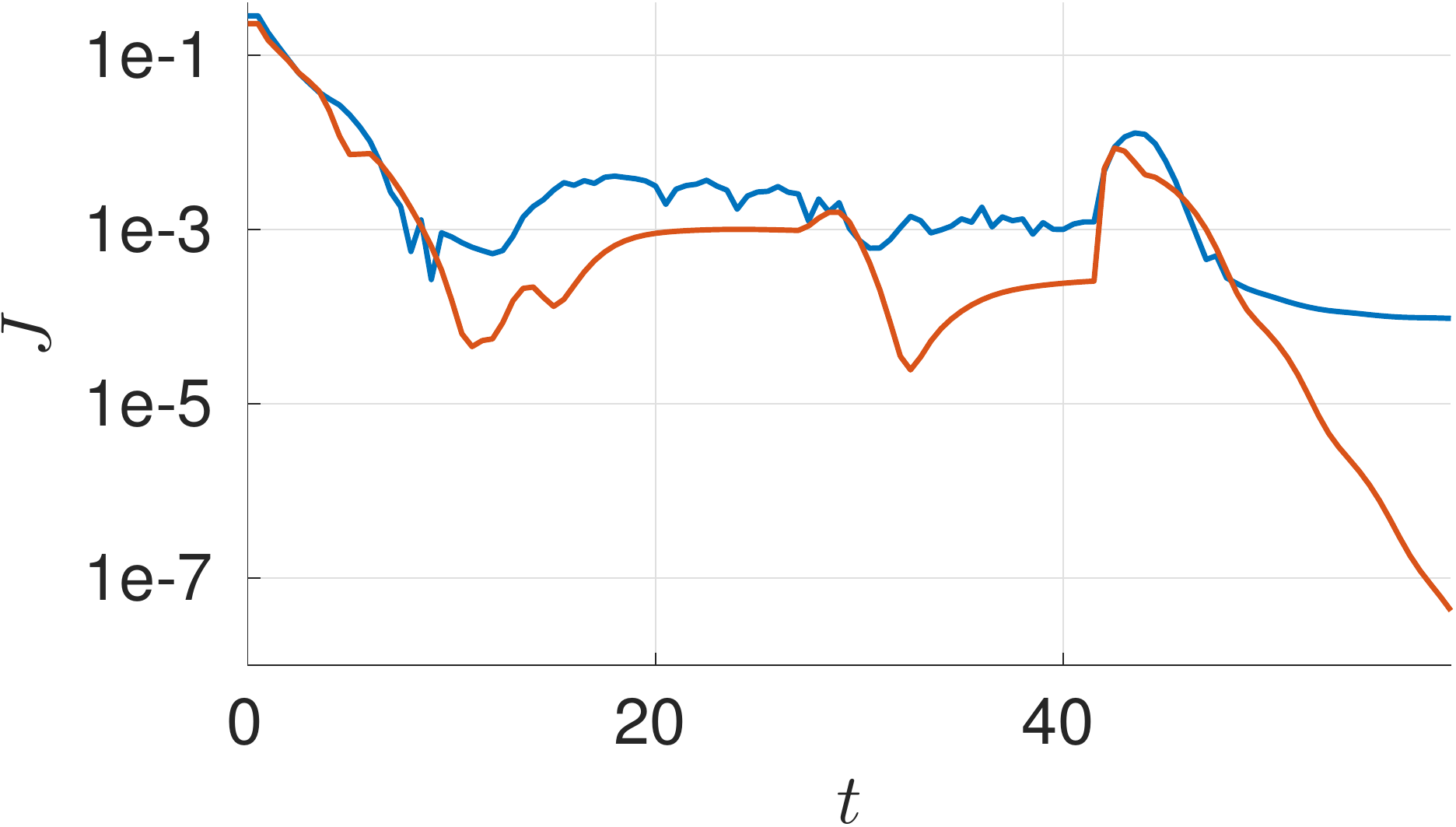}}\hfil
	\parbox[b]{0.45\textwidth}{\centering (d) \\ \includegraphics[width=.4\textwidth]{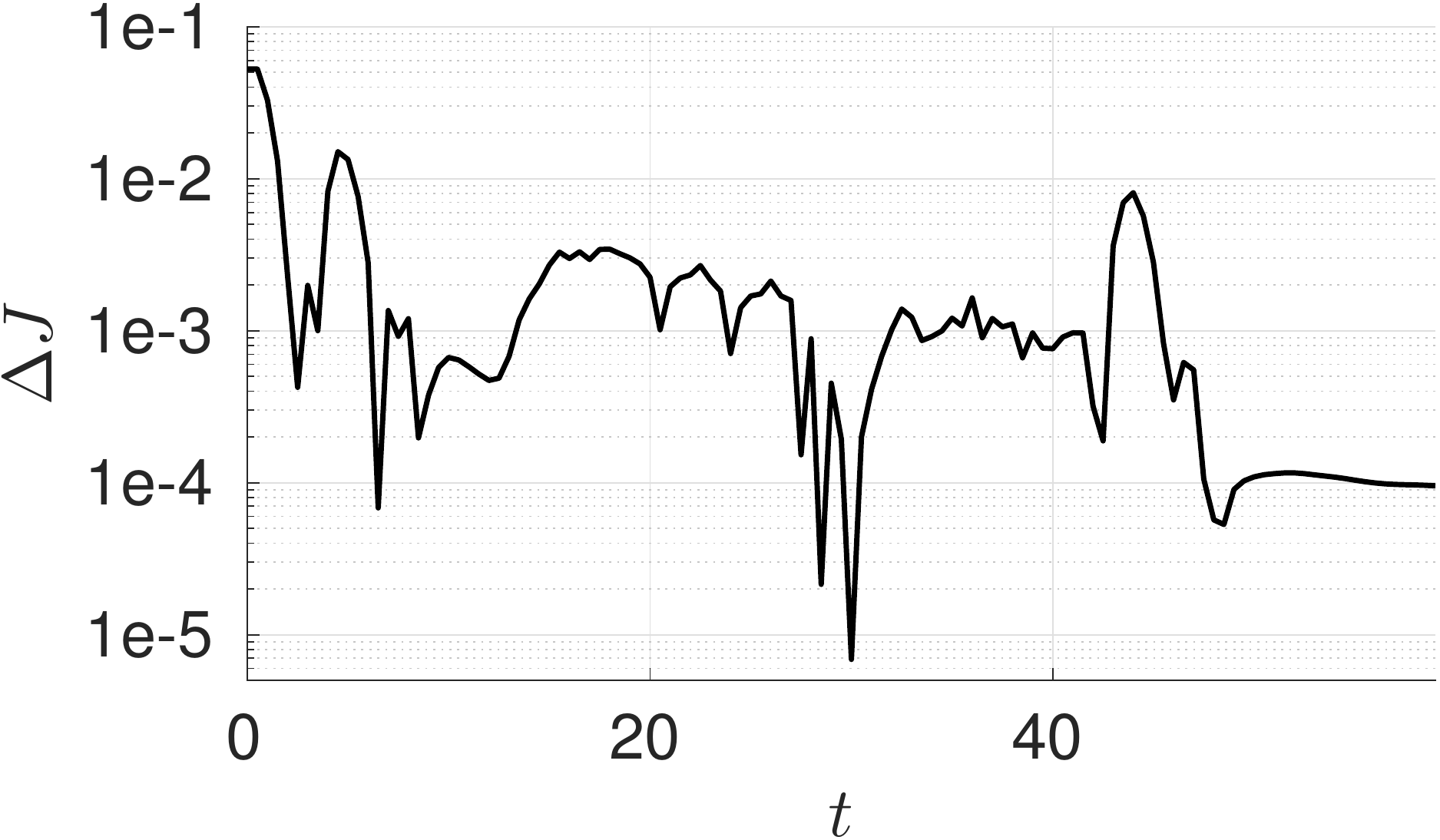}}
	\parbox[b]{0.45\textwidth}{\centering (e) \\ \includegraphics[width=.4\textwidth]{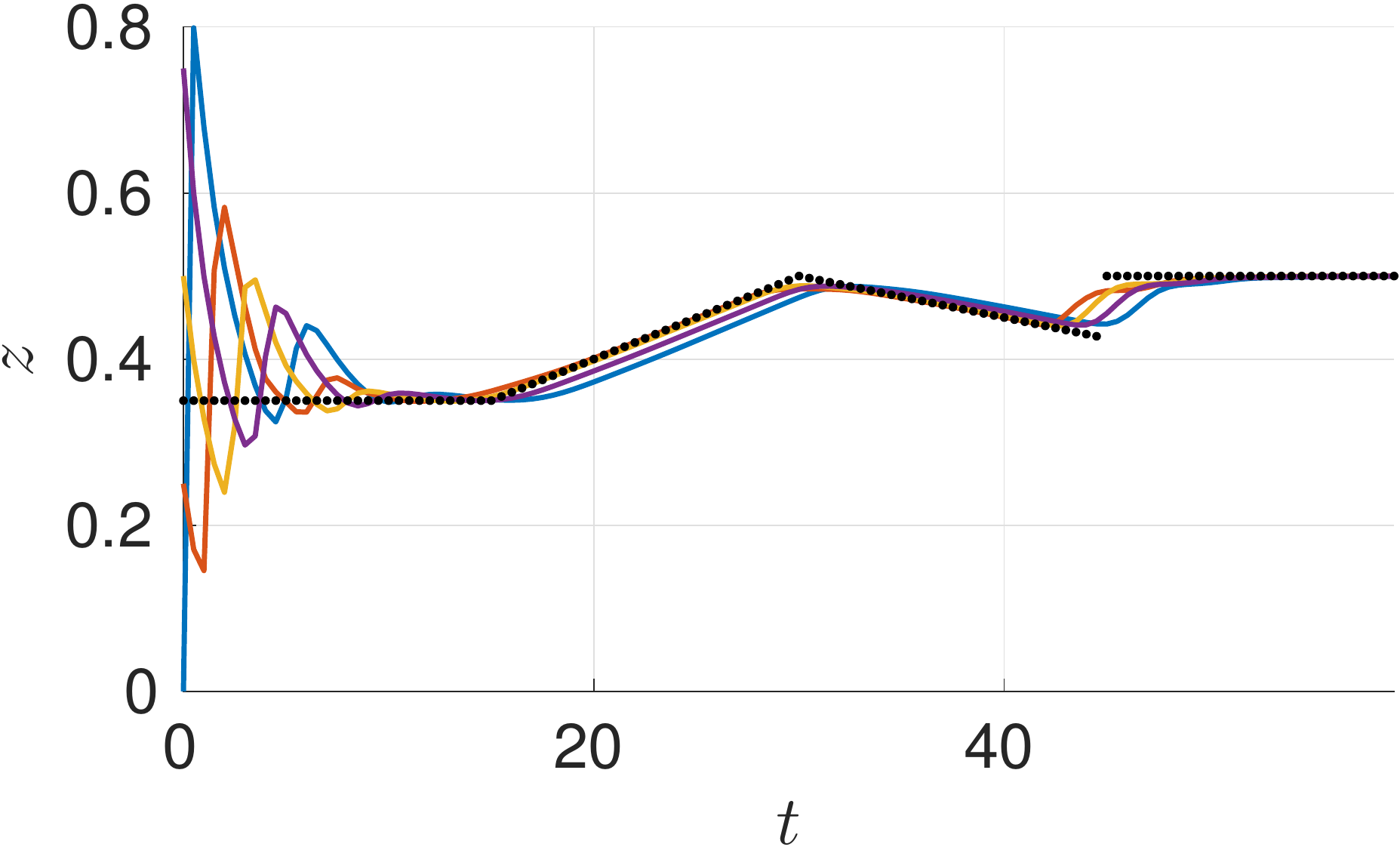}} \hfil
	\parbox[b]{0.45\textwidth}{\centering (f) \\ \includegraphics[width=.4\textwidth]{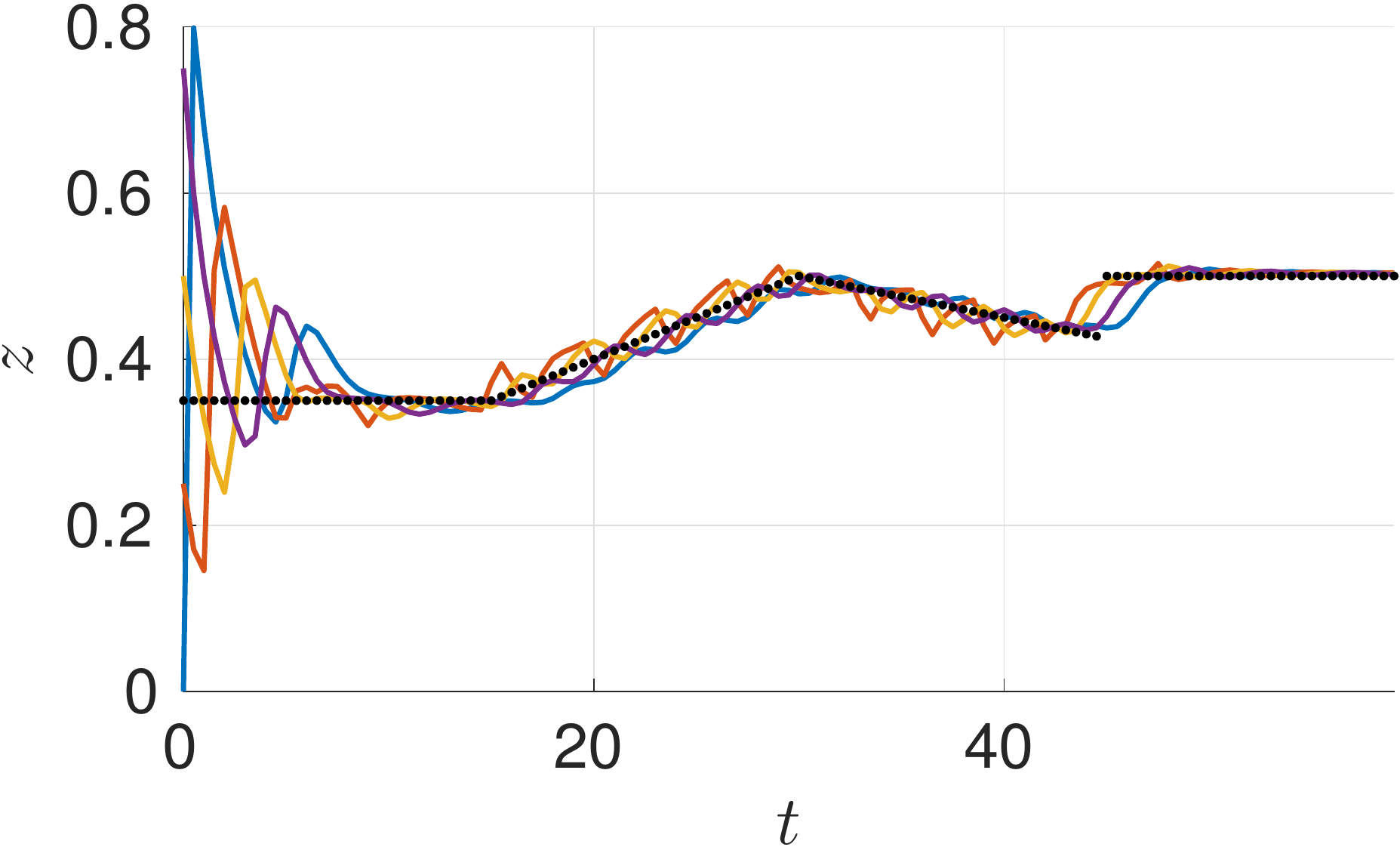}} \\[1ex]
	\parbox[b]{0.6\textwidth}{\centering (g) \\ \includegraphics[width=.52\textwidth]{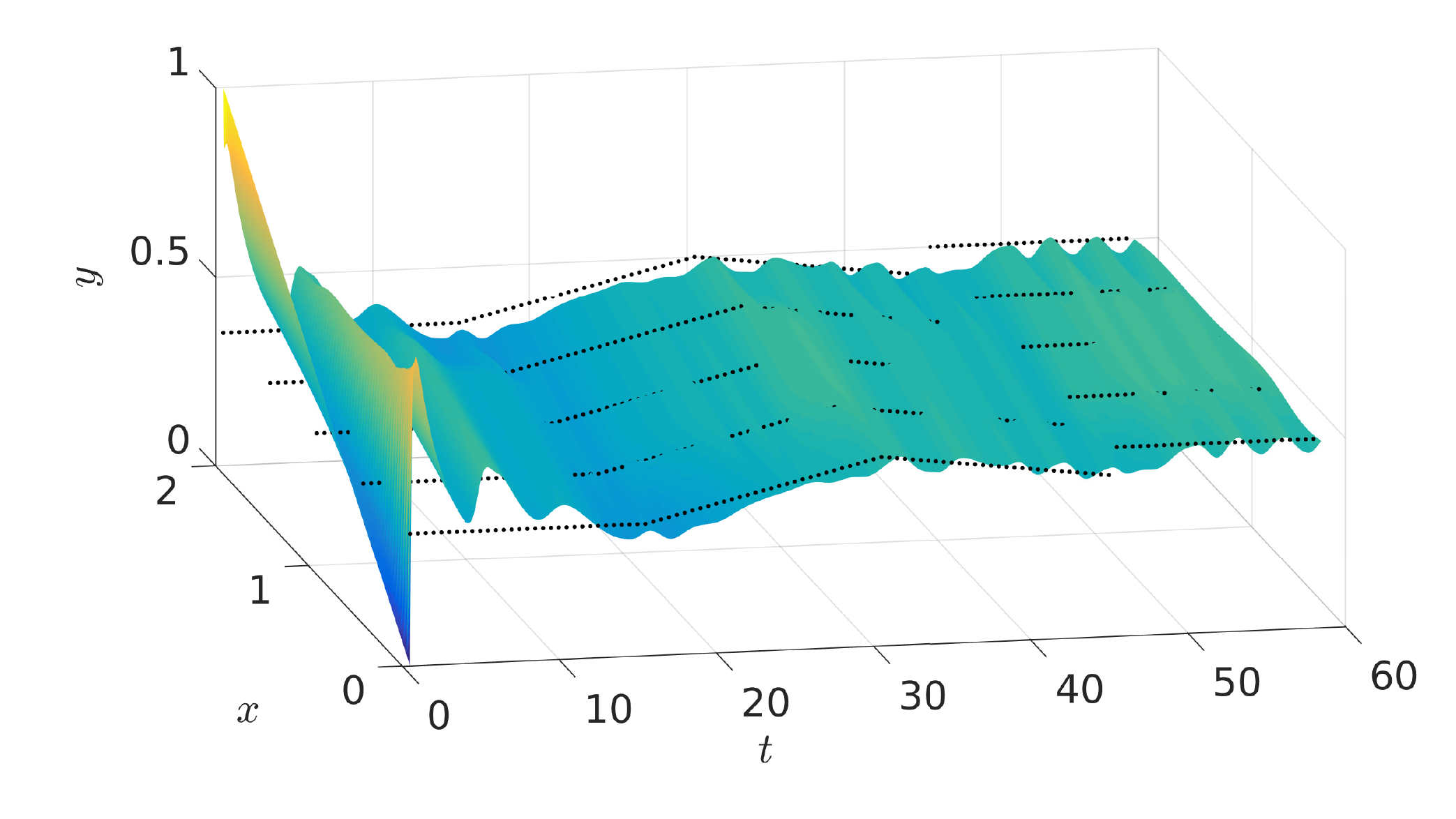}}
	\caption{(a) The shape function $\chi$ and the positions where the state is observed by $f$ for the objective function. (b) Comparison of the optimal control obtained by MPC using the continuous control system \eqref{eq:BurgersContinuous_MPC} (orange) and the switched system \eqref{eq:BurgersSwitched_MPC} (blue). (c) The corresponding objective function values. (d) Difference between the trajectories in (c), $\Delta J = |J_{\eqref{eq:BurgersContinuous_MPC}} - J_{\eqref{eq:BurgersSwitched_MPC}}|$. (e) Optimal trajectory of the observation $\vec{z}$ and the corresponding $\vec{z}^{\mathsf{opt}}$ for the continuous problem \eqref{eq:BurgersContinuous_MPC}. (f) Like (e) but for the switching problem \eqref{eq:BurgersSwitched_MPC}. (g) Optimal state trajectory $y$ corresponding to (e). The reference points $\vec{z}^{\mathsf{opt}}$ are shown in black.}
	\label{fig:Burgers}
\end{figure}

In order to enable comparability to our K-ROM approach, the objective function only depends on a few points in space (the black dots in Figure~\ref{fig:Burgers}~(a)), i.e.,
\begin{equation*}
	\vec{z}(t) = f(y(\cdot,t)) = \left(y(0,t), y(0.5,t), y(1,t), y(1.5,t)\right)^{\top},
\end{equation*}
and we formulate the tracking type objective function in terms of these observations only. This yields the following MPC problem:
\begin{equation}\label{eq:BurgersContinuous_MPC}
	\begin{aligned}
		\min_{\vec{u} \in \R^p} &\sum_{i=s}^{s+p-1} J_i = \|\vec{z}_i - \vec{z}_{i}^{\mathsf{opt}} \|_2^2 \\
		\text{s.t.}\quad y_{i+1} &= \Phi(y_i, \vec{u}_{i-s+1}), \quad i = s,\ldots,s + p - 1, \\
		y_s &= y^s,
	\end{aligned}
\end{equation}
with $\Phi(y, \vec{u})$ being the flow map of \eqref{eq:BurgersContinuous} and $y^s$ the initial value for the current iteration. Similar to Example~\ref{ex:VdP}, we now limit the control input to three constant values $u^0 = 0.075$, $u^1 = 0$ and $u^2 = -0.075$. 
By this, the right-hand side of \eqref{eq:BurgersContinuous} is transformed to $u^j\chi(x)$ which yields $\Phi_{u^j}(y)$ as the flow maps corresponding to the constant inputs $u^0$, $u^1$ and $u^2$. 
The MPC problem is transformed accordingly:
\begin{equation}\label{eq:BurgersSwitched_MPC}
	\begin{aligned} 
		\min_{\vec{\tau} \in \{u^0,u^{1},u^{2}\}^p} &\sum_{i=s}^{s+p-1} J_i = \|\vec{z}_i - \vec{z}_{i}^{\mathsf{opt}} \|_2^2 \\
		\mbox{s.t.}\quad y_{i+1} &= \Phi_{\vec{\tau}_{i-s+1}}(y_i), \quad i = s,\ldots,s + p - 1, \\
		y_s &= y^s,
	\end{aligned}
\end{equation}

The results for the two MPC problems \eqref{eq:BurgersContinuous_MPC} and \eqref{eq:BurgersSwitched_MPC} are compared in Figure~\ref{fig:Burgers}. We see in (b) that the switching approach results in inputs which are close to the continuous solution when averaging over small time frames. In (c) and (d) we see that this results in a solution with slightly reduced quality, i.e., the distance to the reference trajectory is slightly larger. Figures (e) and (f) show the observation $\vec{z}$ for the two problems, i.e., the dynamics which the controller acts on. Finally, the state $y$ corresponding to (e) is shown in (g).

\subsubsection{Switched system MPC based on K-ROMs}
Assuming that we have computed the $n_c$ K-ROMs (i.e., $\vec{U}_{u^0}$ to $\vec{U}_{u^{n_c}}$) via EDMD from data, we can now replace the expensive PDE evaluation in \eqref{eq:MPC_STO} by the much cheaper K-ROM:
\begin{equation}\label{eq:MPC_STO_Koopman} \tag{$\mbox{K-MPC}^{s}$}
	\begin{aligned} 
		&\min_{\hat{\vec{\tau}} \in \{u^0,\ldots,u^{n_c-1}\}^p} \sum_{i=s}^{s+p-1} \hat{L}(\psi(\vec{z}_i)) \\
		\mbox{s.t.}\quad \psi(\vec{z}_{i+1}) &= \vec{U}_{\hat{\vec{\tau}}_{i-s+1}}^{\top} \psi(\vec{z}_{i}) \quad \text{for}~i = s,\ldots,s + p - 1, \\
		\vec{z}_{s} &= f(\vec{y}^s),
	\end{aligned}
\end{equation}
where $\hat{L}$ is the reduced objective function formulated with respect to the observables. 
We now compare the two problem formulations~\eqref{eq:MPC_STO} and \eqref{eq:MPC_STO_Koopman}, i.e., the switching time MPC problem and the corresponding approximation using the K-ROM. To this end, we first assume that the full objective function $L$ can be expressed in terms of the observations $\vec{z}$:
\begin{assumption} \label{ass:equality_objectives_STO}
	$L(\vec{y}(t)) = \hat{L}(\psi(\vec{z}_{i}))$ for all $t\in[t_0,t_0 + h, t_0 + 2h, \ldots, t_e]$ and the corresponding $i = (t-t_0)/h$.
\end{assumption}
The assumption is not restrictive since this has to be satisfied in every application, where only observations are available (e.g., sensor data). Consequently, the objective $L$ has to be defined accordingly.

This assumption allows us to prove convergence for the K-ROM based MPC problem as the number of measurements $m$ and the basis size $k$ for the dictionary $\Psi$ tend to infinity (i.e., we have convergence for EDMD, see \cite{KM17} for details).

\begin{theorem}[\cite{PK17}]\label{thm:Convergence_MPC}
	Consider Problem \eqref{eq:MPC_STO} and the K-ROM based approximation \eqref{eq:MPC_STO_Koopman} and assume that we have convergence of EDMD towards the Koopman operator according to \cite{KM17}, i.e.,
	\[
		\lim_{k,m \rightarrow \infty} \vec{U}_{u^j}^\top = \Kop_{u^j} \quad \mbox{for } j = 1,\ldots, n_c.
	\]
	Then, under Assumption~\ref{ass:equality_objectives_STO}, the optimal solutions of \eqref{eq:MPC_STO} and \eqref{eq:MPC_STO_Koopman} are identical, i.e.,
	\begin{equation*}
		\hat{\vec{\tau}}^{*} = \vec{\tau}^*.
	\end{equation*}
\end{theorem}

As we have seen in Table~\ref{tab:Sampling} for two example problems, this reduced order modeling approach allows us to reduce the numerical effort by several orders of magnitude. A more detailed discussion regarding the numerical benefits will follow in Section~\ref{sec:InfluenceDataBasis}.

\subsection{Bilinear models via linear interpolation}
\label{subsec:KROM_Bilinear}

The switched systems approach allows us to replace the system dynamics by a K-ROM in a straightforward manner. However, this approach comes at a prize, namely that we now have a combinatorial optimization problem which is often more expensive to solve. Furthermore, we have limited the control input to a small number of predefined values. Although this does not necessarily have a major impact on the control performance as we have seen in Section~\ref{subsec:KROM_Switched}, we would nonetheless like to allow for arbitrary control inputs. One possibility to do so is to approximate a Koopman operator for an augmented observation, i.e., 
\[
	\hat{\vec{z}} = \left(\begin{array}{c}
		\vec{z} \\ u
	\end{array}\right).
\]
This approach is pursued in \cite{PBK15} for open-loop and in \cite{KM16} for closed-loop control problems. In order to approximate the Koopman operator for $\hat{\vec{z}}$, data for different combinations of states and inputs has to be collected. In order to reduce the amount of required data, we here pursue an alternative approach where we still only approximate Koopman operators for a small number of autonomous systems. In order to allow for continuous controls, we now define the matrices
\begin{equation*}
	\vec{A} = \vec{U}^{\top}_{u^0} \quad \mbox{and} \quad \vec{B} = \vec{U}^{\top}_{u^1} - \vec{U}^{\top}_{u^0}
\end{equation*}
and introduce the bilinear control system
\begin{equation} \label{eq:KROM_continuous} 
	\begin{aligned}
	\psi(\vec{z}_{i+1}) &= \vec{A} \psi(\vec{z}_{i}) + \vec{B} \psi(\vec{z}_{i}) \frac{u_i - u^0}{u^1 - u^0}, \\
	\vec{z}_0 &= \vec{z}^0 = f(\vec{y}^0).
	\end{aligned}
\end{equation}
The term bilinear refers to the fact that \eqref{eq:KROM_continuous} contains a term $\psi(\vec{z}) \cdot u$ but is otherwise linear both in $\psi(\vec{z})$ and in $u$ \cite{Ell09}. We set the lower and upper bound for the input to $u^0$ and $u^1$, respectively. Consequently, for $u_i \in [u^0,u^1]$, we simply interpolate linearly between the two autonomous dynamics corresponding to $u^0$ and $u^1$. Provided that we have convergence of EDMD, system \eqref{eq:KROM_continuous} is equal to the observation of the exact dynamics for $u^0$ and $u^1$. In order to show convergence for intermediate values, we have to introduce another assumption, namely that the flow map $\Phi(\vec{y},u)$ depends linearly on $u$.
\begin{lemma}[\cite{Pei18}]\label{lem:Koopman_continuous}
	Consider a dynamical control system of the form 
	\begin{align*}
		\vec{y}_{i+1} &= \Phi(\vec{y}_{i},u_i), \\
		\vec{y}_0 &= \vec{y}^0,
	\end{align*}
	and let $\Kop_{u^0}$ and $\Kop_{u^1}$ be the Koopman operators associated with the constant control inputs $u^0$ and $u^1$.
	
	Assume that the observation map $f$ is linear and that the system dynamics $\Phi$ are linear in $u$. Then, a linear interpolation between the two operators is equal to the Koopman operator for the linear interpolation between the controls $u^0$ and $u^1$, i.e.,
	\begin{equation*}
		\alpha \Kop_{u^0} + (1 - \alpha) \Kop_{u^1} = \Kop_{\alpha u^0 + (1-\alpha) u^1}, \quad \alpha \in [0,1].
	\end{equation*}
\end{lemma}
\begin{proof}
	The claim follows directly from the linearity assumptions:
	\begin{align*}
		\big(\left(\alpha \Kop_{u^0} + (1 - \alpha) \Kop_{u^1}\right) f\big) \vec{y}&= \big(\alpha \Kop_{u^0} f\big) \vec{y} + \big((1 - \alpha) \Kop_{u^1} f\big) \vec{y} \\
		&= \alpha f(\Phi(\vec{y}, u^0)) + (1-\alpha) f(\Phi(\vec{y}, u^1)) \\
		&= f\big(\alpha \Phi(\vec{y}, u^0) + (1-\alpha) \Phi(\vec{y}, u^1)\big) \\
		&= f\big(\Phi(\vec{y}, \alpha u^0 + (1-\alpha) u^1)\big) \\
		&= \big(\Kop_{\alpha u^0 + (1-\alpha) u^1} f\big)\vec{y}. 
	\end{align*}
\end{proof}
The above lemma hence yields convergence of the bilinear K-ROM:
\begin{corollary}\label{cor:EqualDynamics}
	The observations of the state of the full dynamical system are equal to the solution of \eqref{eq:KROM_continuous} provided that we have convergence of the EDMD algorithm, i.e.,
	\begin{equation*}
		f(\Phi(\vec{y}_{i},u_i)) = P \left( \vec{A} \psi(f(\vec{y}_{i})) + \vec{B} \psi(f(\vec{y}_{i})) \frac{u_i - u^0}{u^1 - u^0} \right).
	\end{equation*}
\end{corollary}
In a similar fashion to the switched systems approach, we can now easily reduce the numerical effort of \eqref{eq:MPC} (i.e., the MPC problem with continuous inputs) by replacing the system dynamics by the bilinear surrogate model:
\begin{equation}\label{eq:MPC_Bilinear_Koopman}\tag{K-MPC}
	\begin{aligned}
		&~\qquad\qquad\qquad\min_{\vec{u} \in \R^p} \sum_{i=s}^{s+p-1} \hat{L}(\psi(\vec{z})_i) \\
		\psi(\vec{z})_{i+1} &= \vec{A} \psi(\vec{z})_{i} + \vec{B} \psi(\vec{z})_{i} \frac{\vec{u}_{i-s+1} - u^0}{u^1 - u^0} \quad \text{for}~i = s,\ldots,s + p - 1, \\
		\vec{z}_s &= f(\vec{y}^s).
	\end{aligned}
\end{equation}
Convergence of the reduced problem now follows immediately from Lemma~\ref{lem:Koopman_continuous}.
\begin{theorem}\label{thm:Convergence_MPC_bilinear} 
	Consider Problem \eqref{eq:MPC} with a control system $\Phi(\vec{y},u)$ which depends linearly on $u$ and the K-ROM based approximation \eqref{eq:MPC_Bilinear_Koopman}. Assume that we have convergence of EDMD towards the Koopman operator according to \cite{KM17}, i.e.,
	\[
	\lim_{k,m \rightarrow \infty} \vec{U}_{u^j}^\top = \Kop_{u^j} \quad \mbox{for } j = 1,\ldots, n_c.
	\]
	Then, under Assumption~\ref{ass:equality_objectives_STO}, Problems \eqref{eq:MPC_Bilinear_Koopman} and \eqref{eq:MPC} possess the same solution.
\end{theorem}
\begin{proof}
	The claim follows directly from Lemma~\ref{lem:Koopman_continuous} and Corollary~\ref{cor:EqualDynamics}.
\end{proof}

\begin{remark}
	\eqref{eq:MPC_Bilinear_Koopman} is a bilinear control problem. Efficient algorithms specifically tailored to this problem class exist, see, e.g., \cite{PY08,Ell09}. An alternative to using MPC would be, for instance, to solve a state-dependent Riccati equation \cite{Cim08} in order to obtain a closed-loop controller.
	\exampleSymbol
\end{remark}

\begin{remark}\label{rem:Local_Bilinear_K-ROMs}
	Lemma \ref{lem:Koopman_continuous} and Theorem \ref{thm:Convergence_MPC_bilinear} are only valid if the control system $\Phi(\vec{y},u)$ depends linearly on $u$ which is not always the case. (In \cite{Pei18}, the influence of nonlinear control dependencies has been studied.) In these situations, a way to reduce the inaccuracy of the bilinear system \eqref{eq:KROM_continuous} is to introduce multiple bilinear K-ROMs which are valid locally, which is inspired by similar concepts in the reduced-basis community, see, e.g., \cite{AHKO12}. Consequently, the linear interpolation is performed between two operators which are less far apart in terms of the control input. This means that we approximate several Koopman operators $\vec{U}_{u^j}$ corresponding to $u^0 < u^1 < \ldots < u^{n_c - 1}$. The K-ROM then consists of several locally valid bilinear models:
	\begin{equation}\label{eq:KROM_local}
		\begin{aligned}
			\big[A,~B\big] = \begin{cases}
			\big[\vec{U}_{u^0},~\vec{U}_{u^1} - \vec{U}_{u^1}\big] & \text{for } u \in [u^0, u^1) \\
			\big[\vec{U}_{u^1},~\vec{U}_{u^2} - \vec{U}_{u^2}\big] & \text{for } u \in [u^1, u^2) \\
			\qquad \qquad \vdots  & \qquad \quad \vdots \\
			\big[\vec{U}_{u^{n_c - 2}},~\vec{U}_{u^{n_c - 1}} - \vec{U}_{u^{n_c - 2}}\big] & \text{for } u \in [u^{n_c - 2}, u^{n_c - 1}].
			\end{cases}
		\end{aligned}
	\end{equation}
	Note that due to this, the control system is continuous and piece-wise smooth with possible kinks at $u^1, u^2,\ldots, u^{n_c - 2}$ which has to be taken into account by the optimization routine solving \eqref{eq:MPC_Bilinear_Koopman}.
	\exampleSymbol
\end{remark}

\section{On the influence of the amount of data and the selection of basis functions}
\label{sec:InfluenceDataBasis}

In this section, we study the two examples already mentioned in Table~\ref{tab:Sampling} in more detail. 
Since the convergence result for EDMD only holds for infinitely large dictionaries $\Psi$ as well as infinitely many data points, the assumptions of the convergence theorems are obviously not satisfied in a practical setting.
This means that we need to investigate the influence of the basis functions used in the construction of the dictionary $\Psi$ as well as the impact of the amount of training data on the controller performance. Furthermore, we compare the two K-ROM approaches against each other and against the full solution. The latter is only possible for the Burgers equation as the numerical effort for solving the Navier--Stokes based MPC problem is prohibitively large.

\subsection{Test cases and reference setup}
Here, we first introduce the two test cases and validate the K-ROM approach using one particular numerical setup. All algorithms except the Navier--Stokes simulations are implemented in Matlab. For the switched-systems optimization, all possible ${n_c}^p$ inputs $\vec{\tau}$ are evaluated as motivated in Section~\ref{subsec:KROM_Switched}. For the bilinear K-ROM, the Matlab function \emph{fmincon} is applied which uses a sequential quadratic programming (SQP, see~\cite{NW06}) approach.

\subsubsection{The 1D Burgers Equation}
\label{subsubsec:Burgers}
We now compare the solutions obtained by the full control problems and their K-ROM approximations, respectively, 
using the problem setup introduced in Section~\ref{subsubsec:ExampleBurgers}. For the switched system approach, we choose the inputs $u^0 = -0.075$, $u^1 = 0.075$, and $u^2 = 0$. For the bilinear surrogate model, we use $u^0$ and $u^1$ to construct $A$ and $B$. For the data collection process, we use three different initial conditions and for each of these, we perform one simulation with constant inputs $u^0$, $u^1$ and $u^2$, respectively. Finally, we perform an additional simulation with a constant switching sequence between the inputs. This yields 12 simulations in total, each of which is 60 seconds long, and we collect a snapshot every $0.005$ seconds. The switched sequences are then split into three matrices according to which input is active during which time step. These data points are then attached to the respective snapshot matrices with constant inputs. 
The time step $h$ for the flow map $\Phi$ in the control problem as well as for the approximation of the Koopman operator is set to $0.5$ seconds.

The performance of the reduced approaches is visualized in Figure~\ref{fig:Burgers_STO_Full_vs_KROM} (for a prediction horizon of length $p=3$) where the switched approaches are compared in the left column and the continuous ones in the right column. 
\begin{figure}[h!]
	\centering
	\parbox[b]{0.45\textwidth}{\centering (a) \\ \includegraphics[width=.43\textwidth]{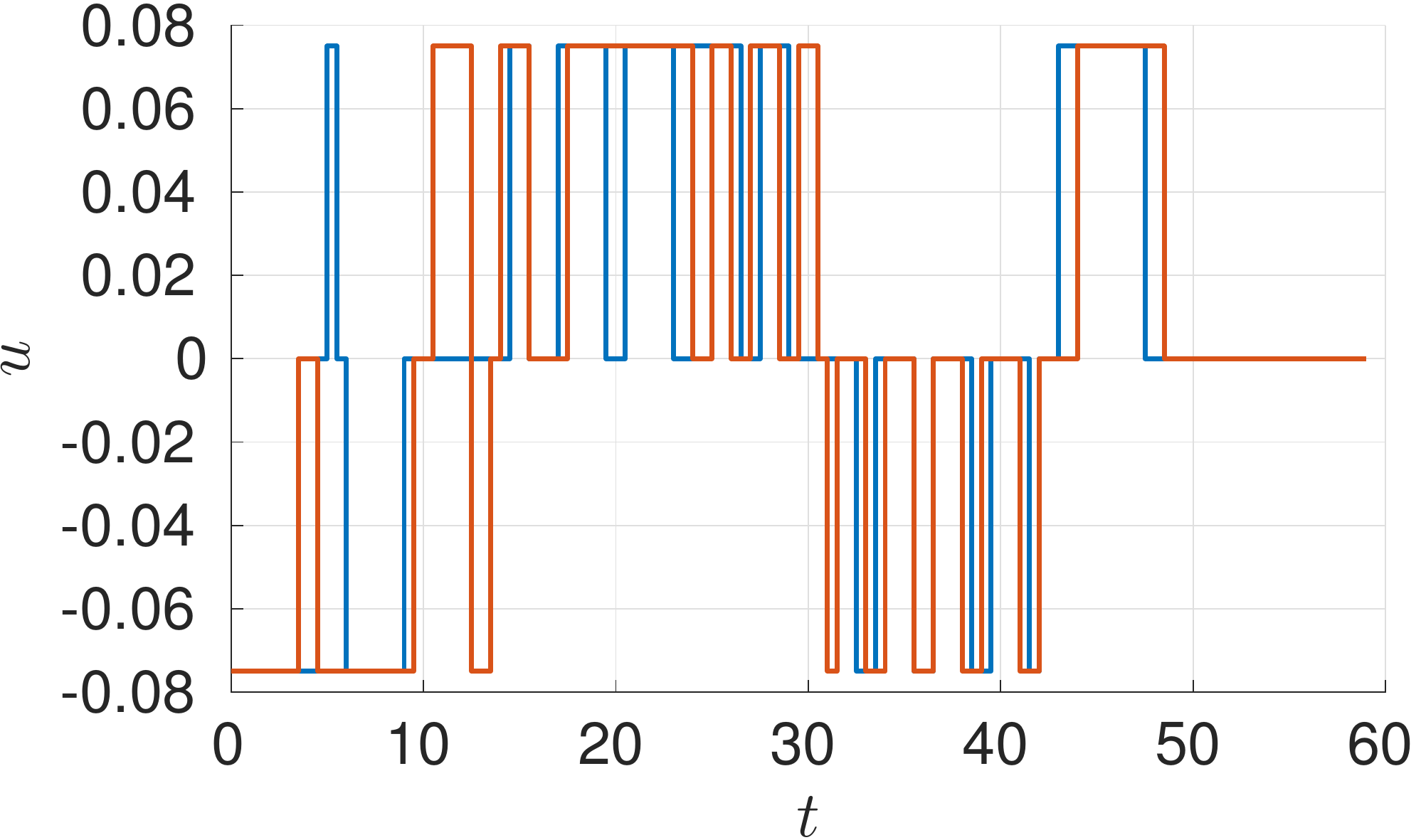}} \hfil
	\parbox[b]{0.45\textwidth}{\centering (b) \\ \includegraphics[width=.43\textwidth]{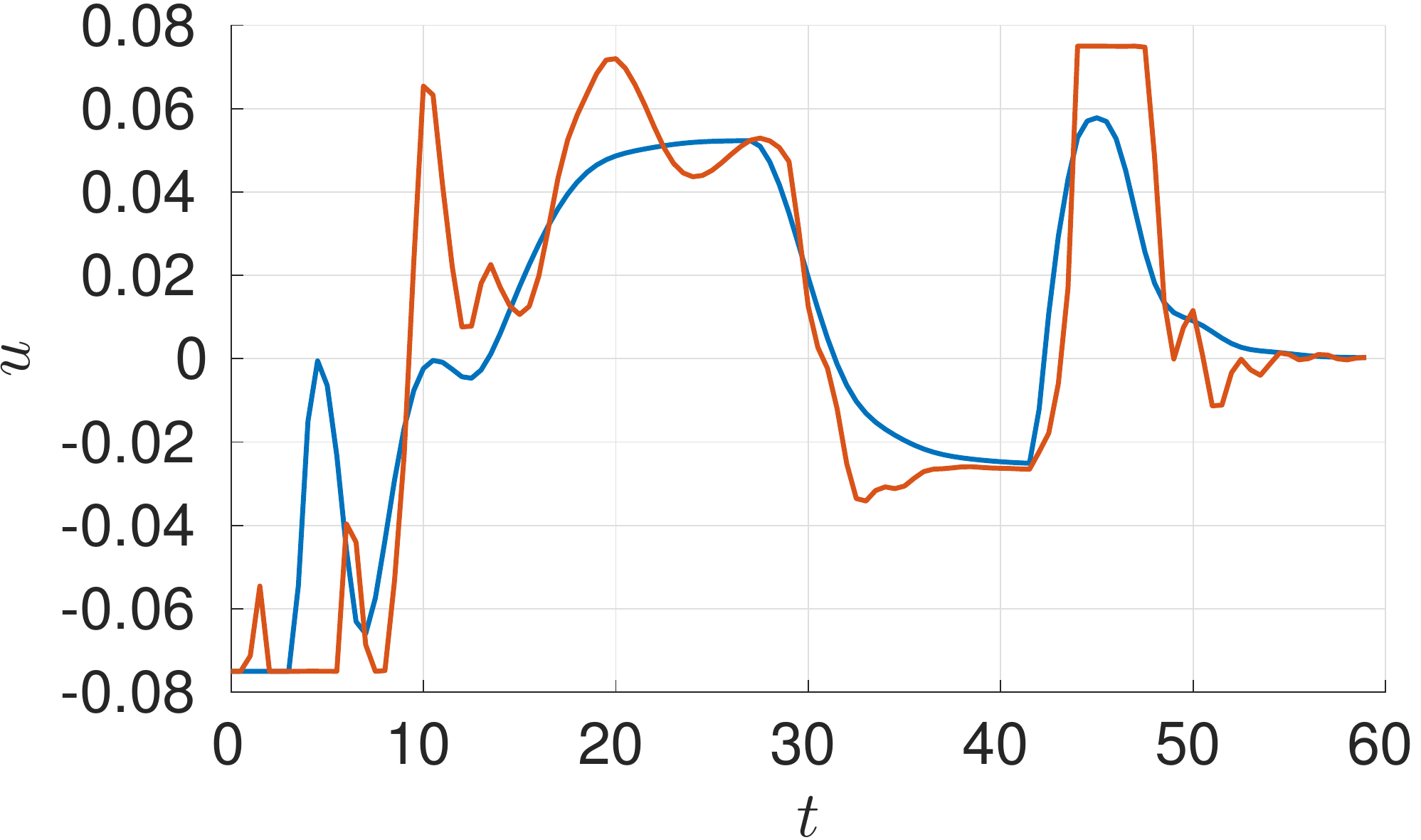}} \\[1ex]
	\parbox[b]{0.45\textwidth}{\centering (c) \\ \includegraphics[width=.43\textwidth]{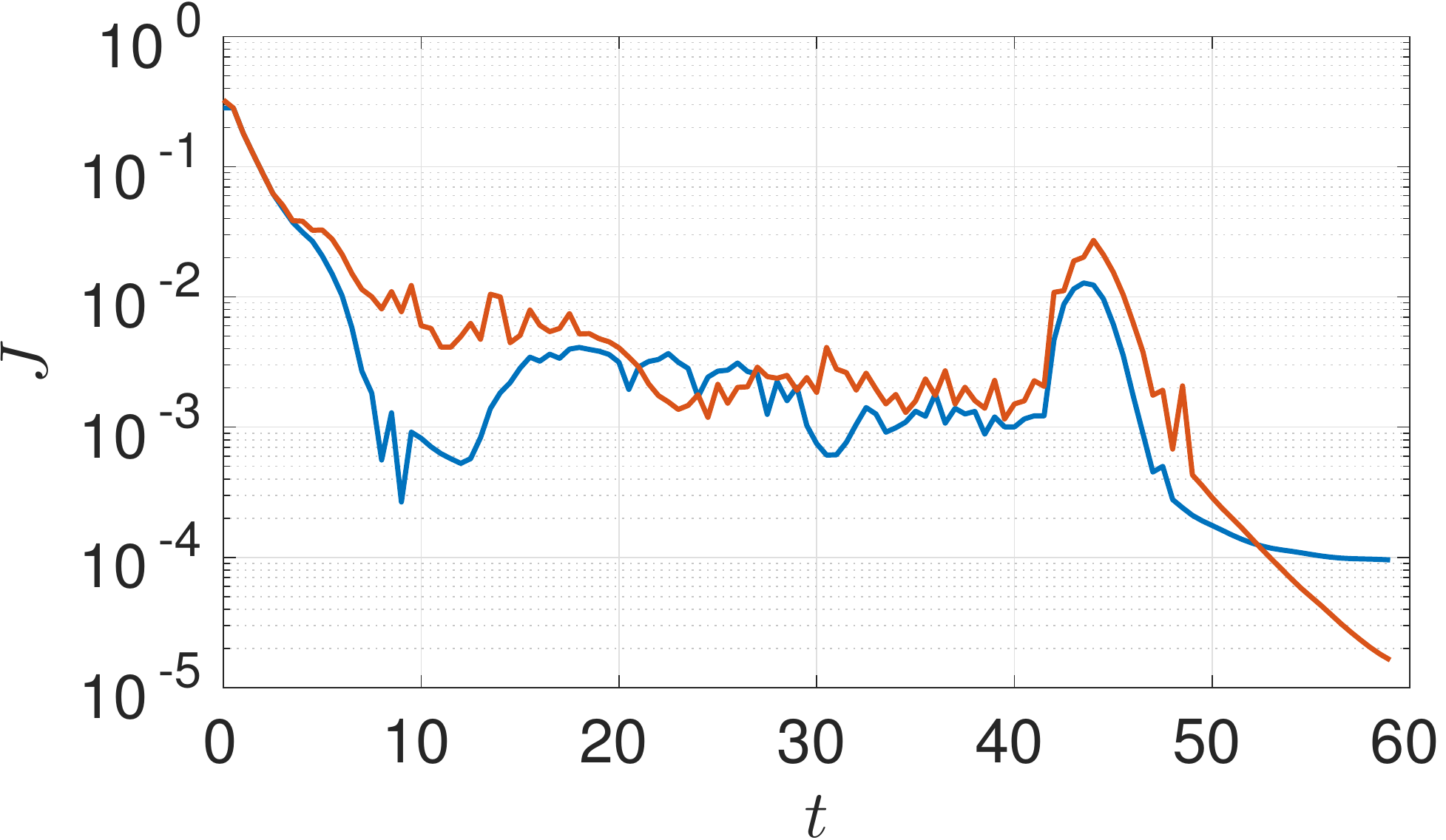}}\hfil
	\parbox[b]{0.45\textwidth}{\centering (d) \\ \includegraphics[width=.43\textwidth]{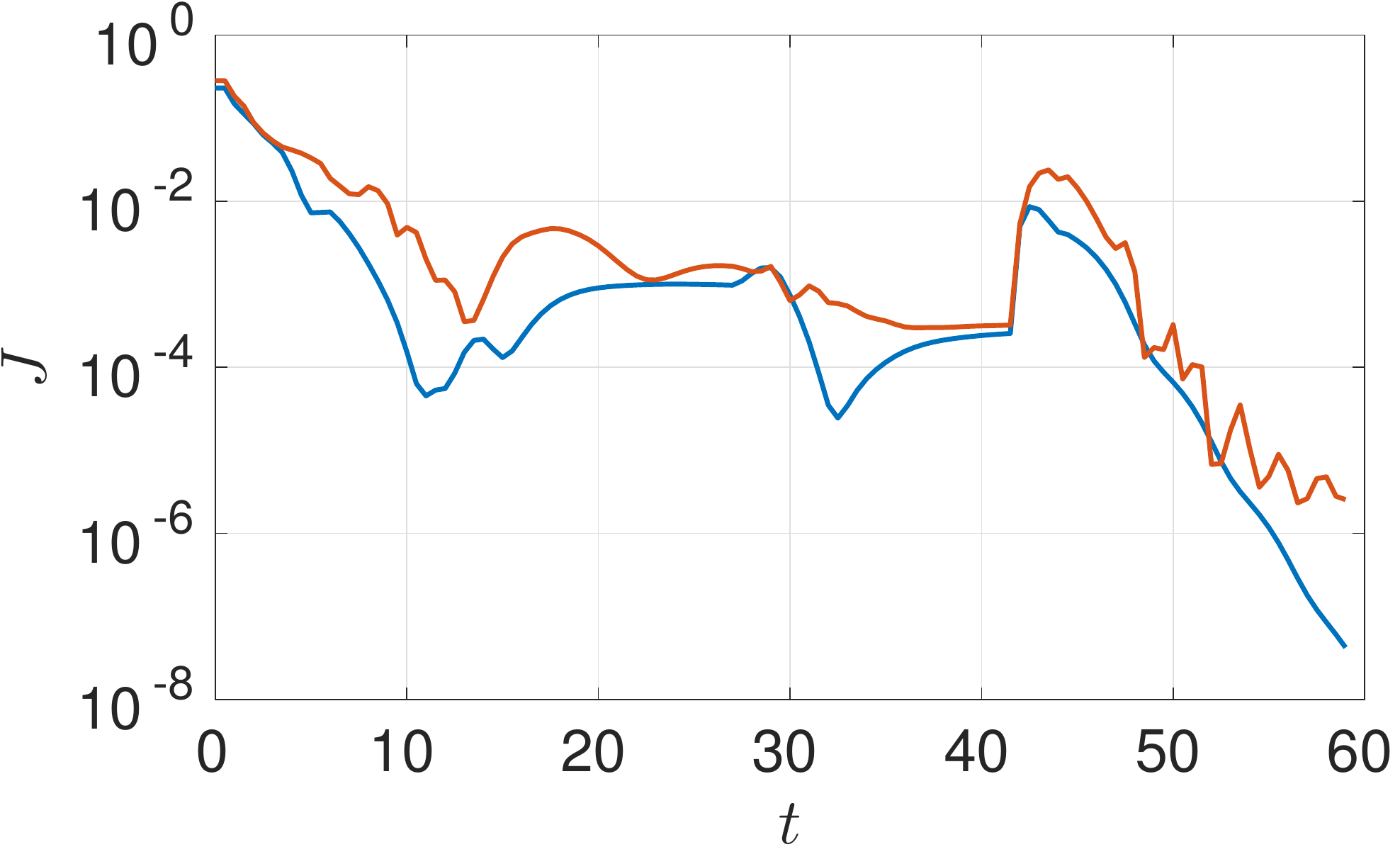}}\\[1ex]
	\parbox[b]{0.45\textwidth}{\centering (e) \\ \includegraphics[width=.43\textwidth]{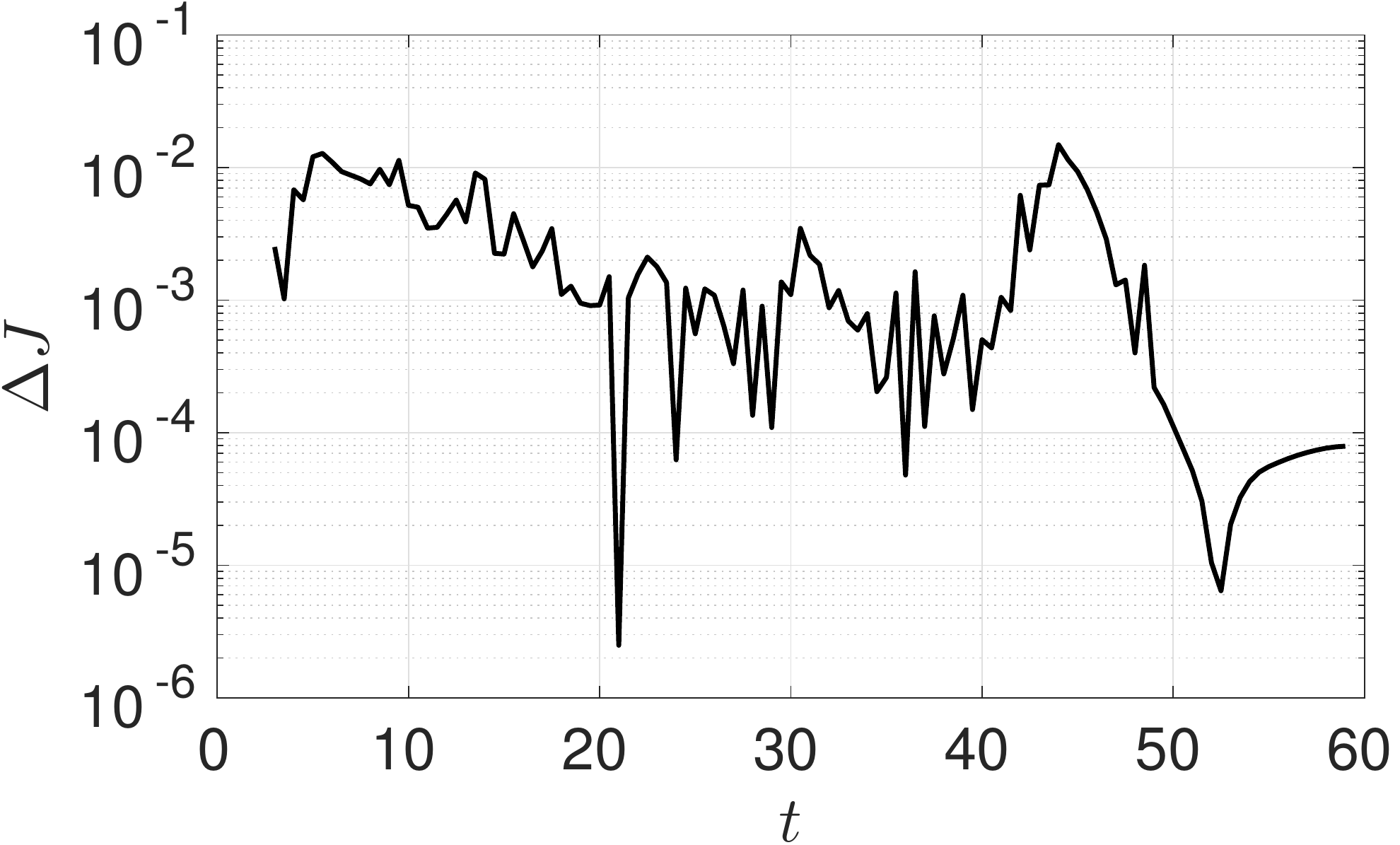}} \hfil
	\parbox[b]{0.45\textwidth}{\centering (f) \\ \includegraphics[width=.43\textwidth]{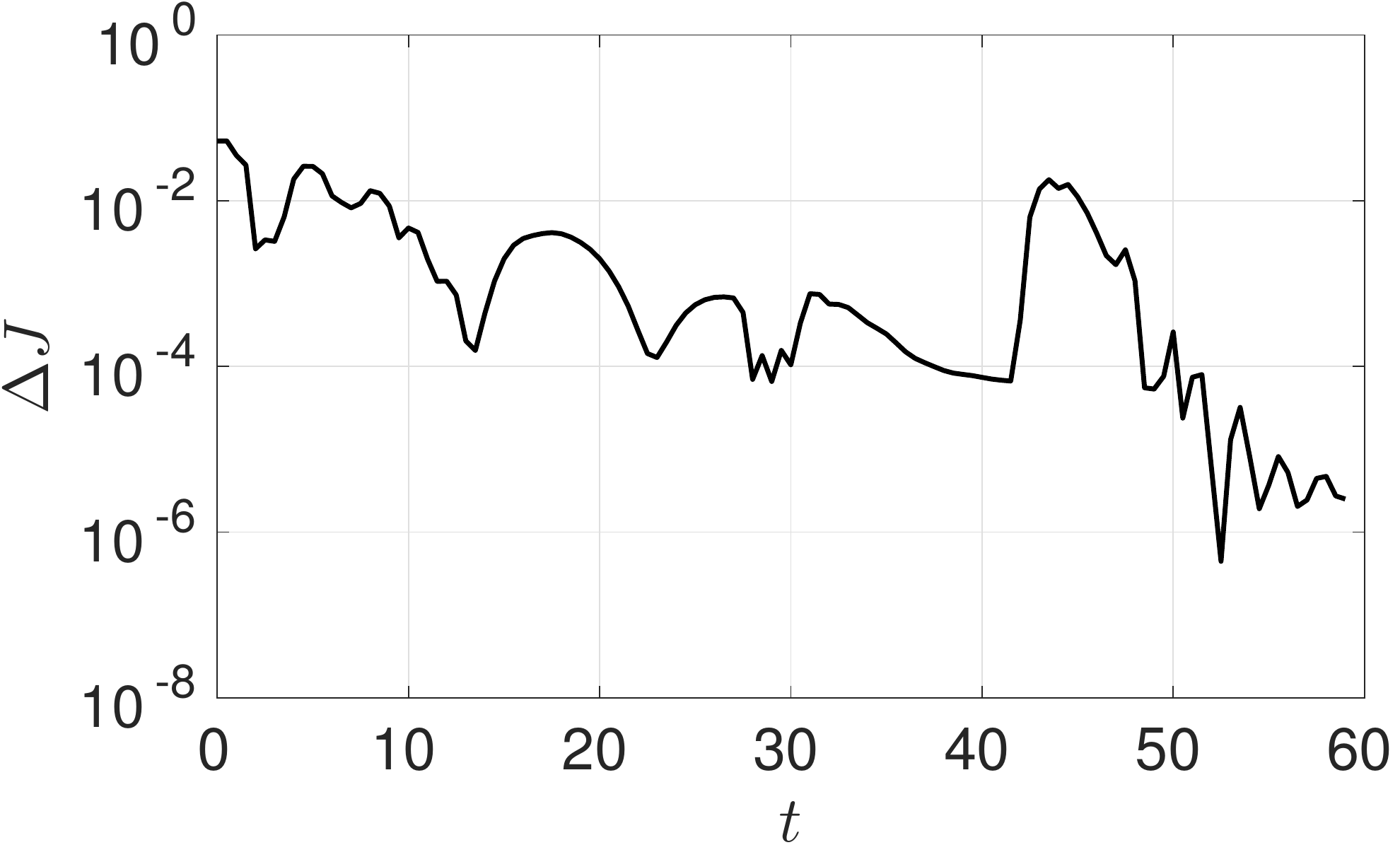}} \\[1ex]
	\parbox[b]{0.45\textwidth}{\centering (g) \\ \includegraphics[width=.43\textwidth]{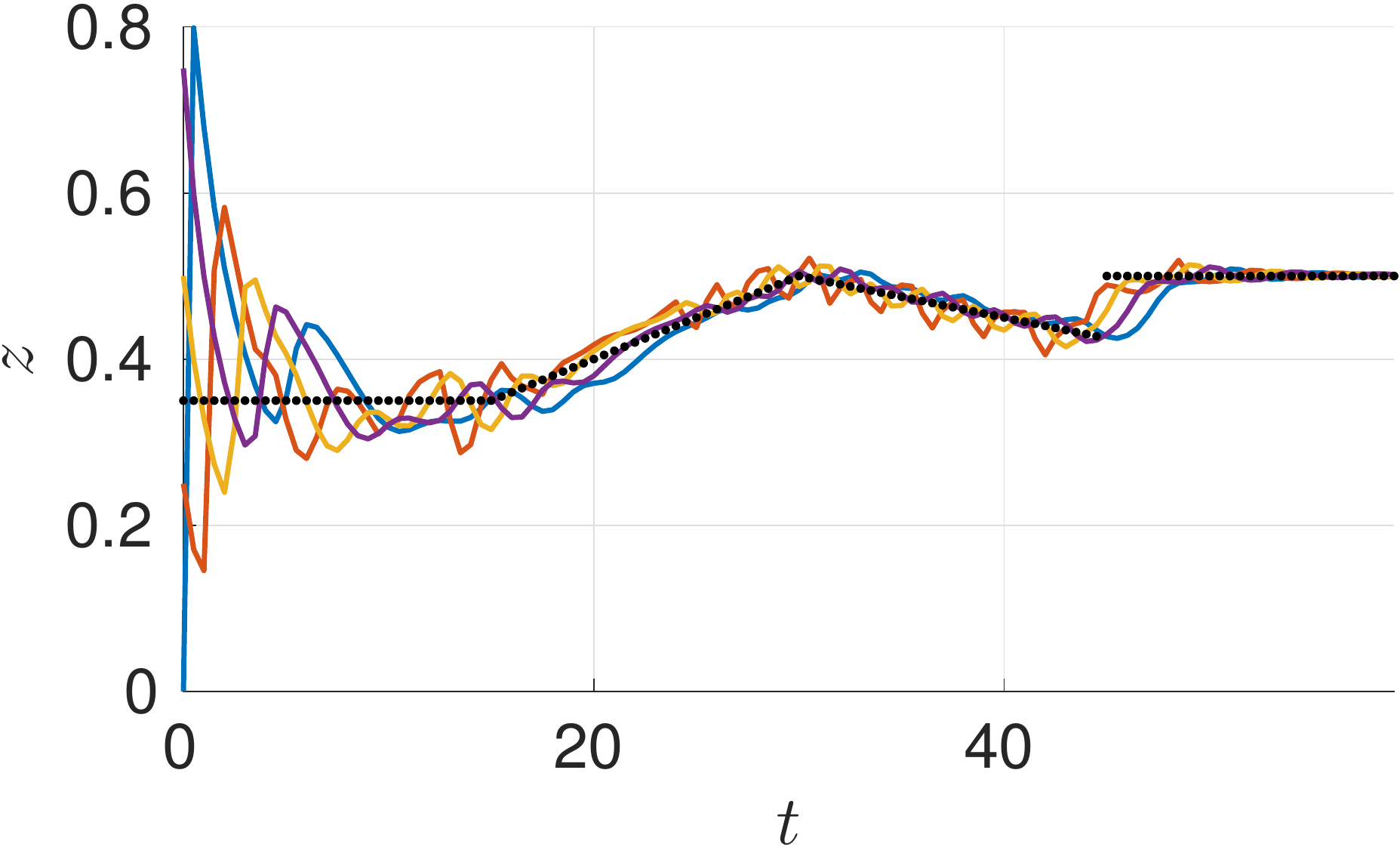}} \hfil
	\parbox[b]{0.45\textwidth}{\centering (h) \\ \includegraphics[width=.43\textwidth]{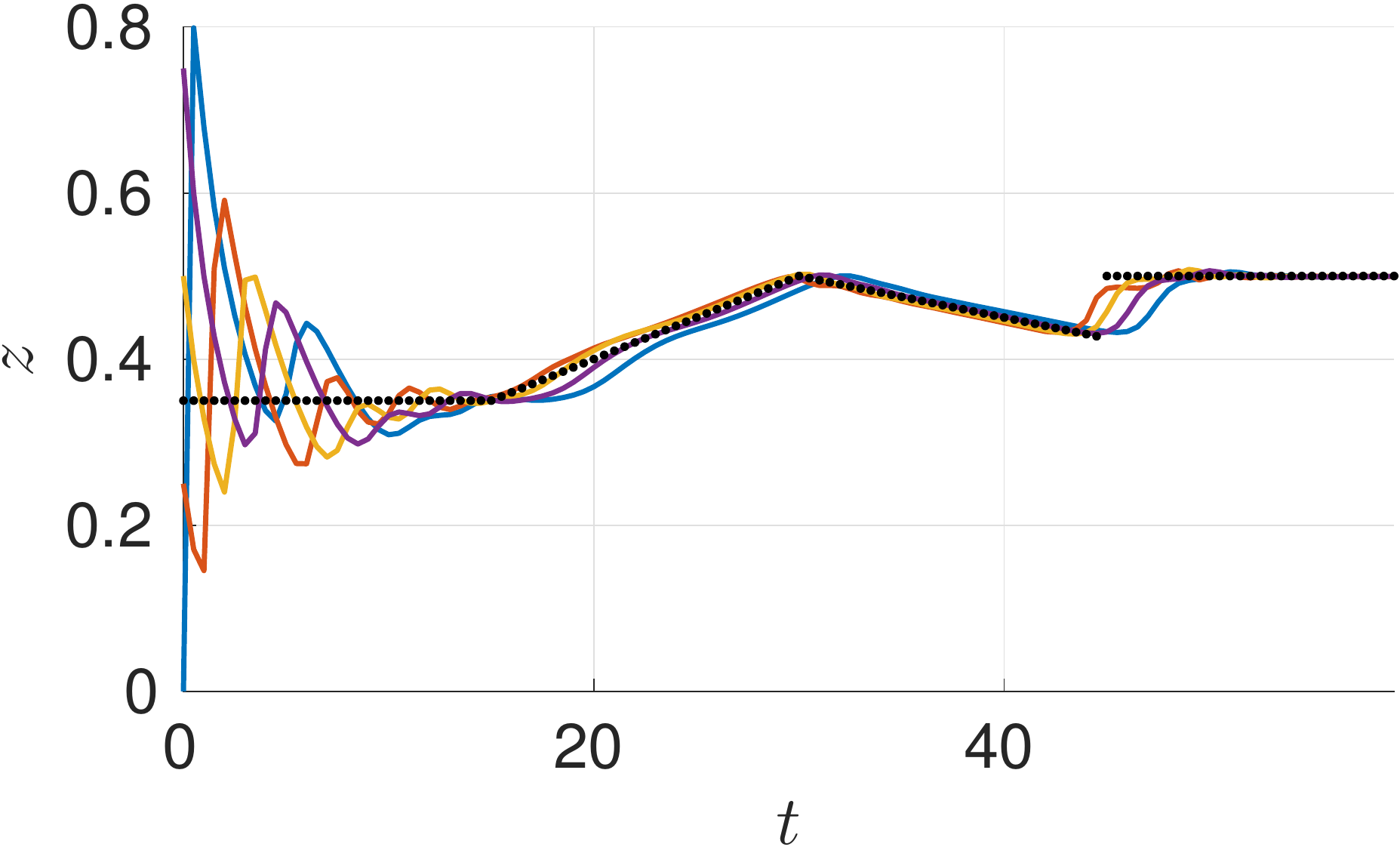}}
	\caption{Left column: \eqref{eq:MPC_STO} vs.~\eqref{eq:MPC_STO_Koopman}. Right column: \eqref{eq:MPC} vs.~\eqref{eq:MPC_Bilinear_Koopman}. (a,b) Optimal control for the full problem (blue) and for the K-ROM based problem (orange). (c,d) The corresponding objective function values. (e,f) Absolute value of the difference between the full and the reduced solution. (g,h) Optimal trajectories of the observation $\vec{z}$. For a comparison to the full system see Figure~\ref{fig:Burgers} (f) and (e).}
	\label{fig:Burgers_STO_Full_vs_KROM}
\end{figure}

We see that in the switched systems approach, the inputs (Figure~\ref{fig:Burgers_STO_Full_vs_KROM}~(a)) vary significantly. 
This is due to the MPC framework. As soon as the two $\vec{z}$ trajectories differ slightly, the corresponding optimization problems do not necessarily possess the same optimal solution any longer. Consequently, small inaccuracies may result in different control trajectories. 
However, we see in Figure~\ref{fig:Burgers_STO_Full_vs_KROM}~(c) and (e) that the value of the objective function is of comparable quality. In some parts (e.g., at $t \approx 10\,s$), the ripples around the reference trajectory (Figure~\ref{fig:Burgers_STO_Full_vs_KROM}~(g)) are slightly larger than in the PDE-constrained case (cf.~Figure~\ref{fig:Burgers} (f) on p.~\pageref{fig:Burgers}). Nevertheless, it can be concluded that the switched systems K-ROM approach is very well suited for real-time control of the Burgers equation.

Looking at the continuous K-ROM, we see that we have an even better performance (cf.~Figure~\ref{fig:Burgers}~(d)), as can be expected due to the larger freedom in choosing the input to the system. We see that the difference between the PDE based and the K-ROM based solutions is similar to the switched systems case. However, a significant advantage is that we now require only data for two autonomous systems instead of three. This means that the data requirements can be further reduced by $33 \%$. For the same reasons as in the switched systems case, we do not have a very good agreement between the optimal control. We will further study this effect in Section~\ref{subsec:data_sampling}.

\subsubsection{The 2D Navier--Stokes equations}

The second example is the flow around a cylinder described by the 2D incompressible Navier--Stokes equations at a Reynolds number of $Re = 100$:
\begin{align*}
	\dot{\vec{y}}(\vec{x},t) + \vec{y}(\vec{x},t) \cdot \nabla \vec{y}(\vec{x},t) &= \nabla p(\vec{x},t) + \frac{1}{Re} \Delta \vec{y}(\vec{x},t), \\
	\nabla \cdot \vec{y}(\vec{x},t) &= 0, \\
	\vec{y}(\vec{x},0) &= \vec{y}^0(\vec{x}),
\end{align*}
see Figure~\ref{fig:vonKarman}~(a) for the problem setup and a snapshot of the solution computed with \emph{OpenFOAM} \cite{JJT07} using a finite volume discretization with $22,000$ cells. 
The system is controlled via rotation of the cylinder, i.e., $u(t)$ is the angular velocity. Without control, the well-known \emph{von K\'{a}rm\'{a}n vortex street} occurs.
\begin{figure}[b!]
	\centering
	\parbox[b]{0.49\textwidth}{\centering (a) \\ \includegraphics[width=.49\textwidth]{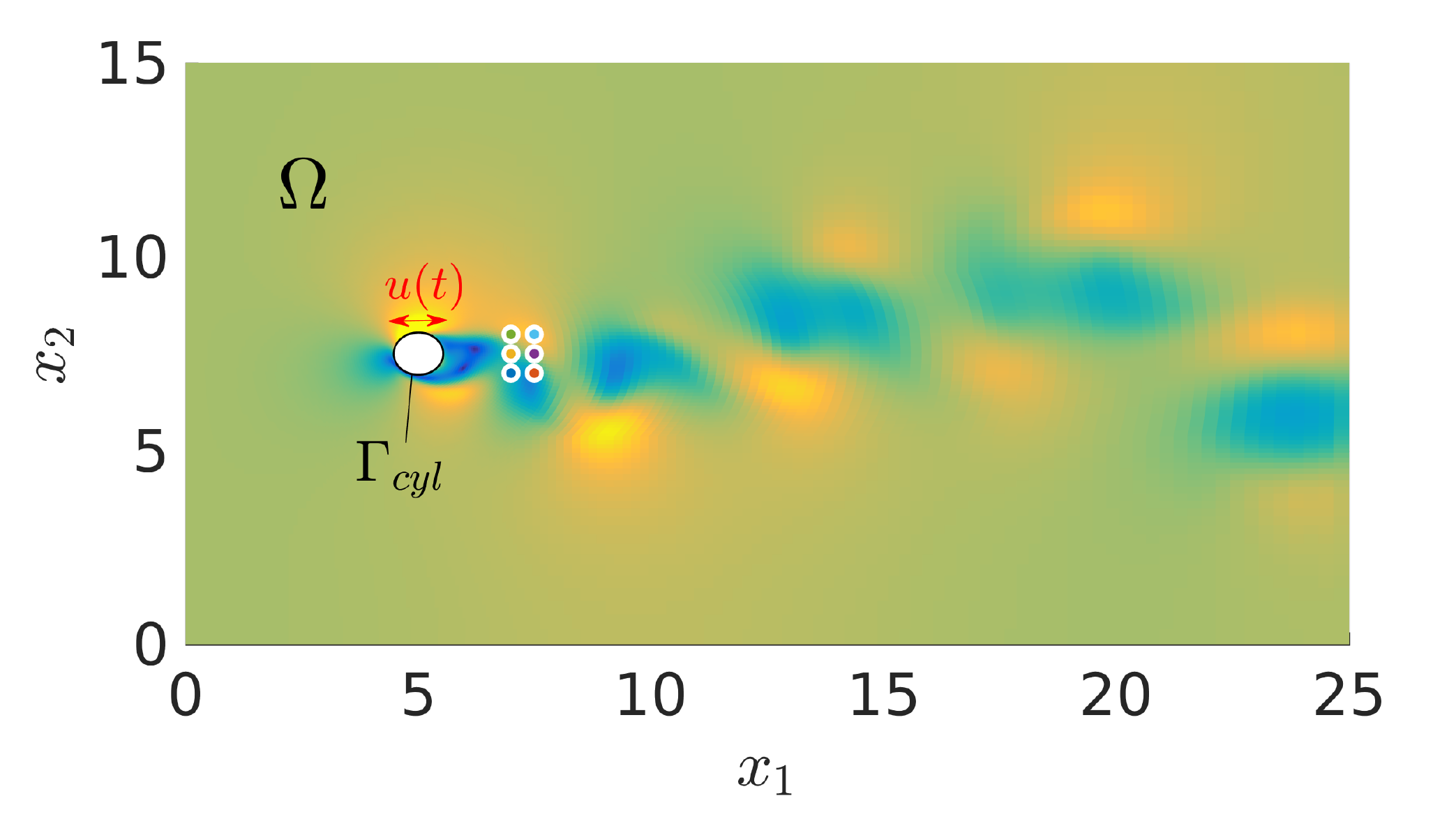}} \hfil
	\parbox[b]{0.49\textwidth}{\centering (b) \\ \includegraphics[width=.47\textwidth]{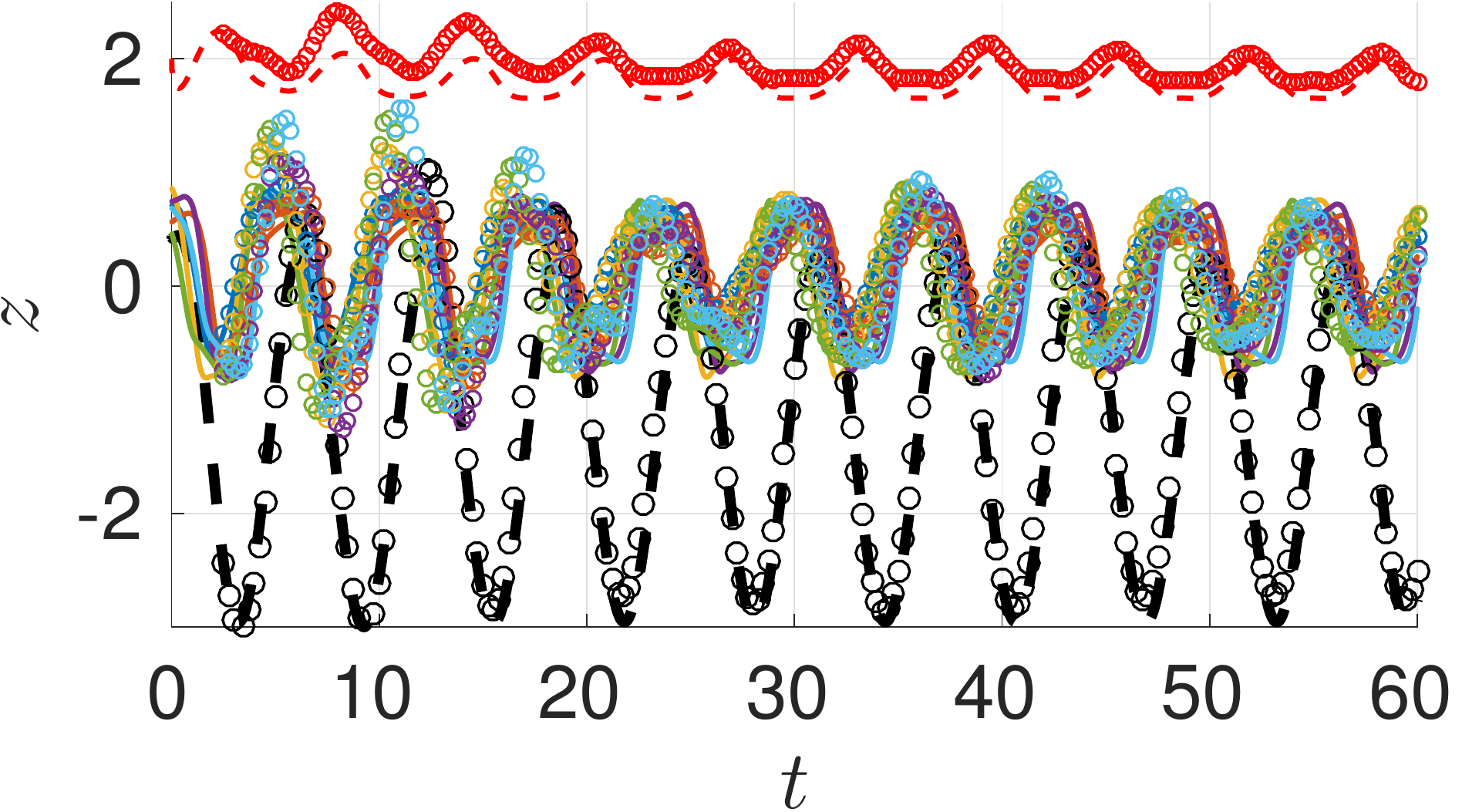}}\\[1ex]
	\parbox[b]{0.49\textwidth}{\centering (c) \\ \includegraphics[width=.49\textwidth]{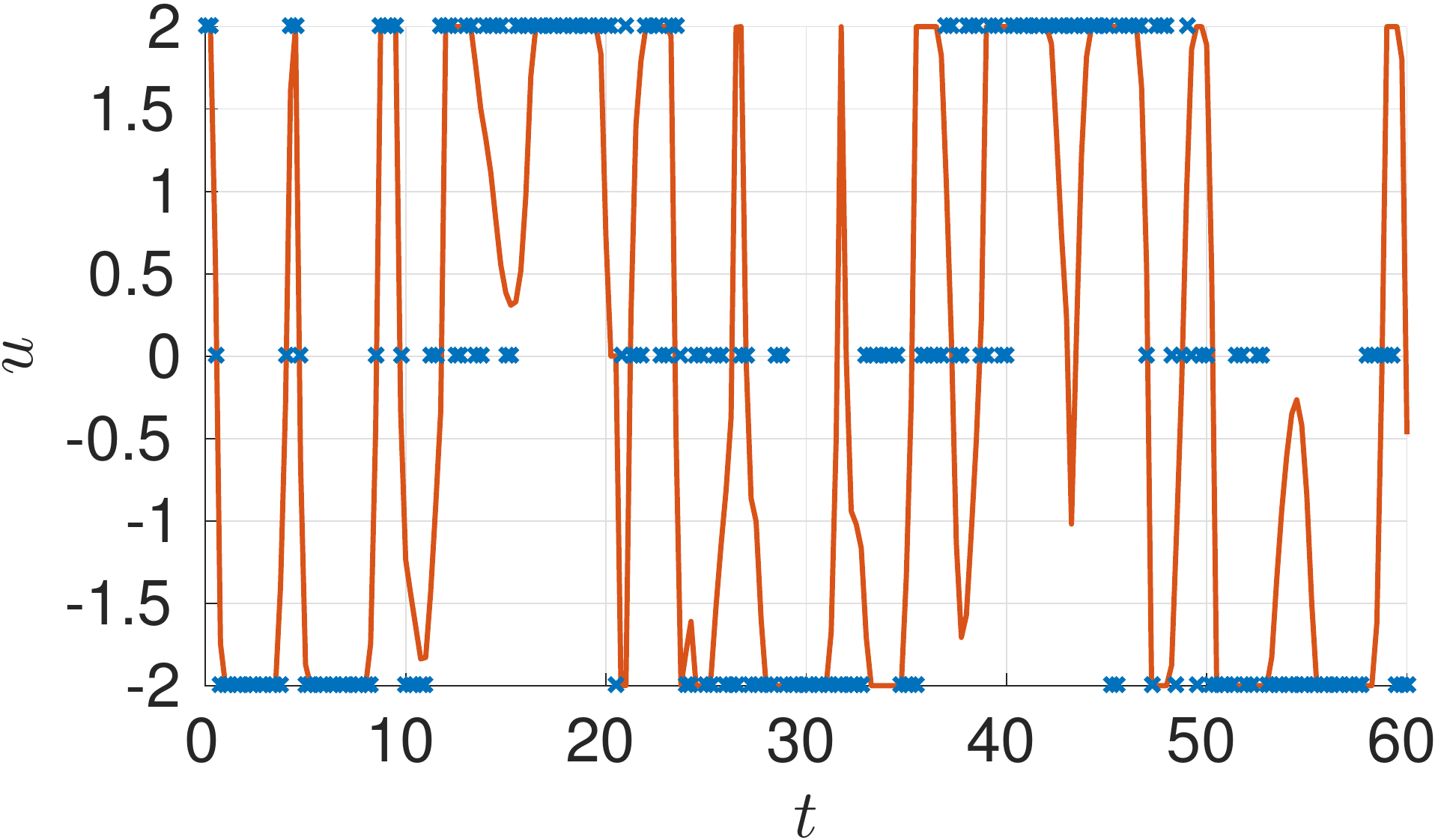}}\hfil
	\parbox[b]{0.49\textwidth}{\centering (d) \\ \includegraphics[width=.49\textwidth]{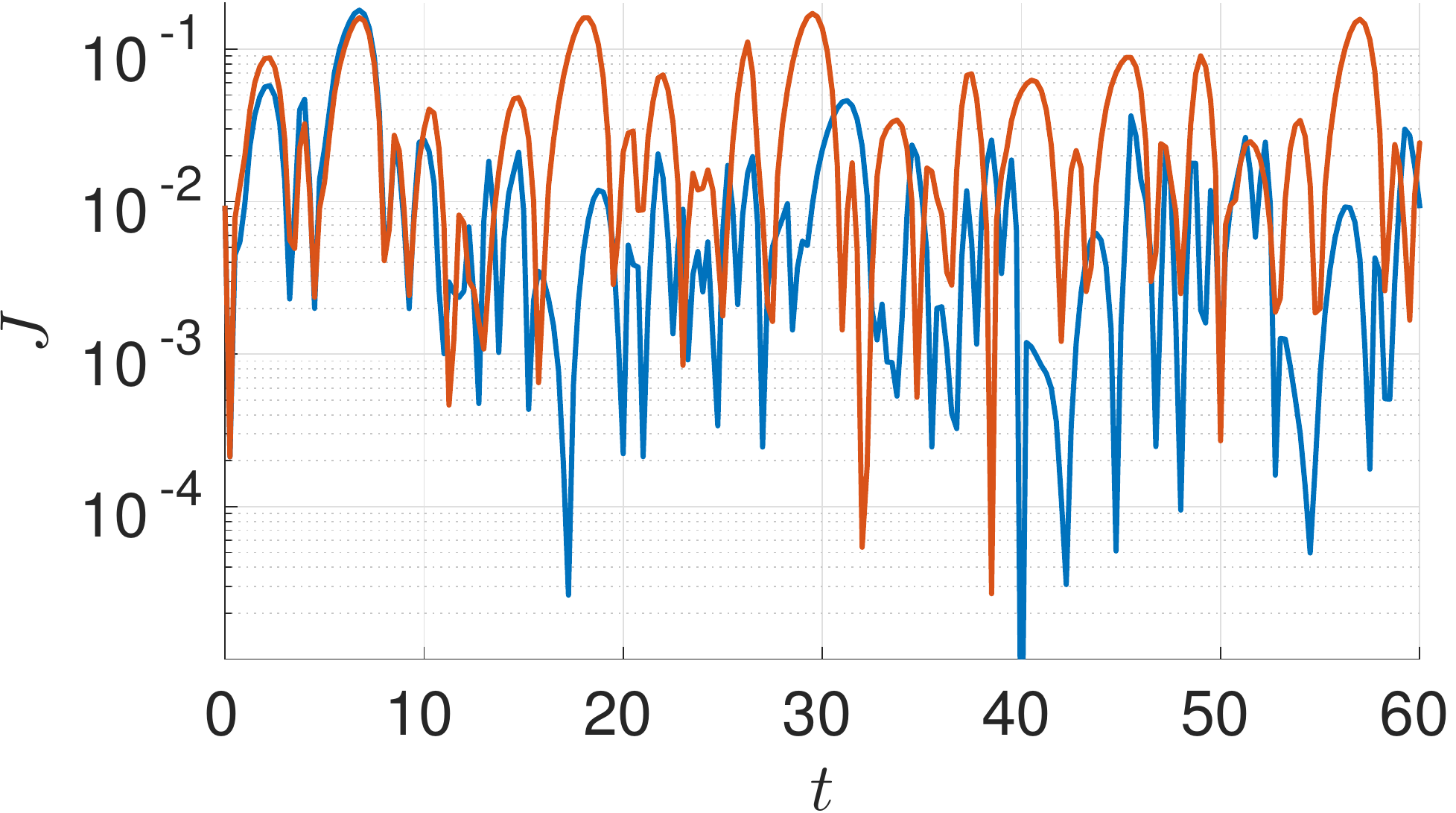}}\\[1ex]
	\parbox[b]{0.49\textwidth}{\centering (e) \\ \includegraphics[width=.49\textwidth]{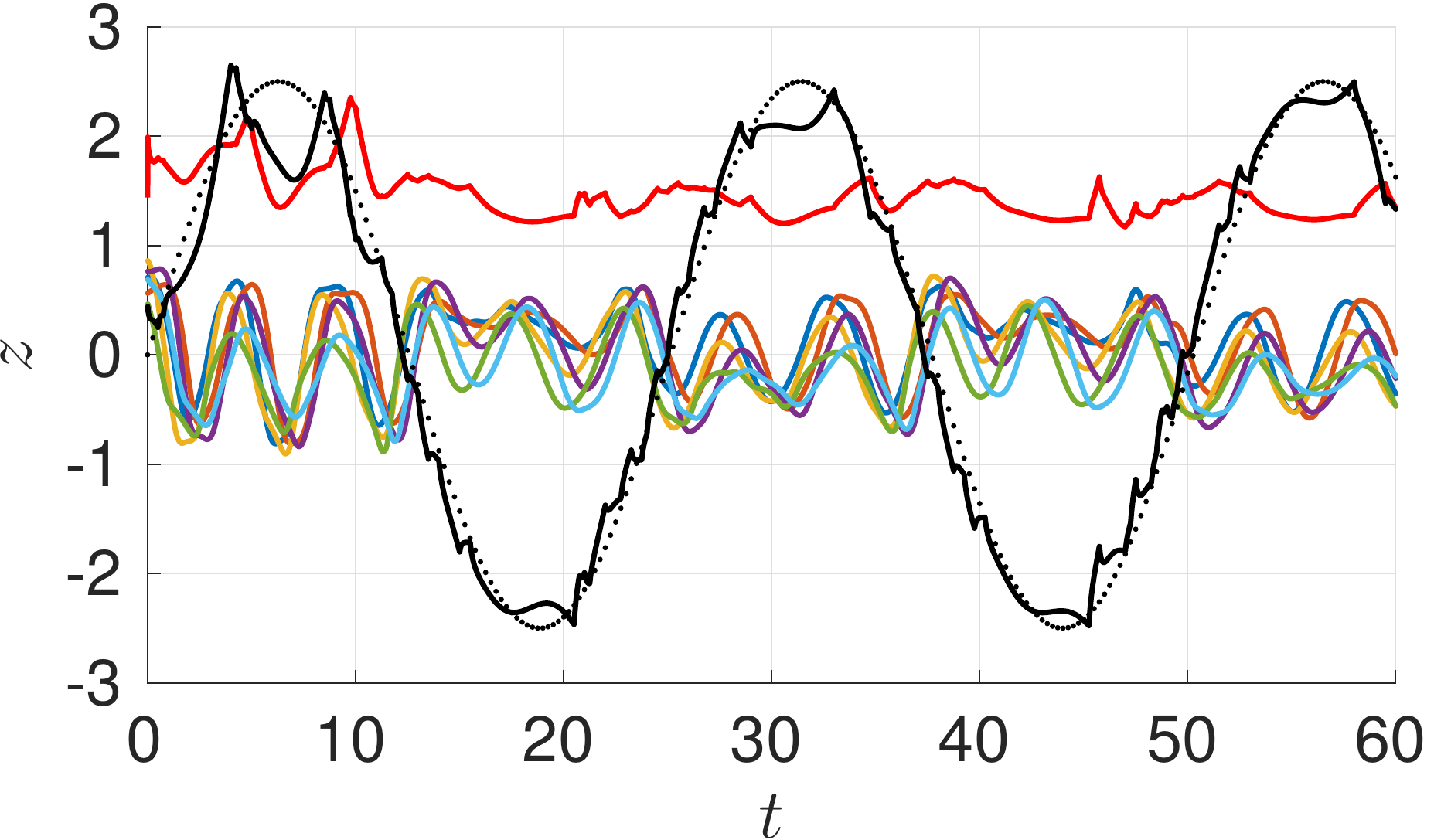}}\hfil
	\parbox[b]{0.49\textwidth}{\centering (f) \\ \includegraphics[width=.49\textwidth]{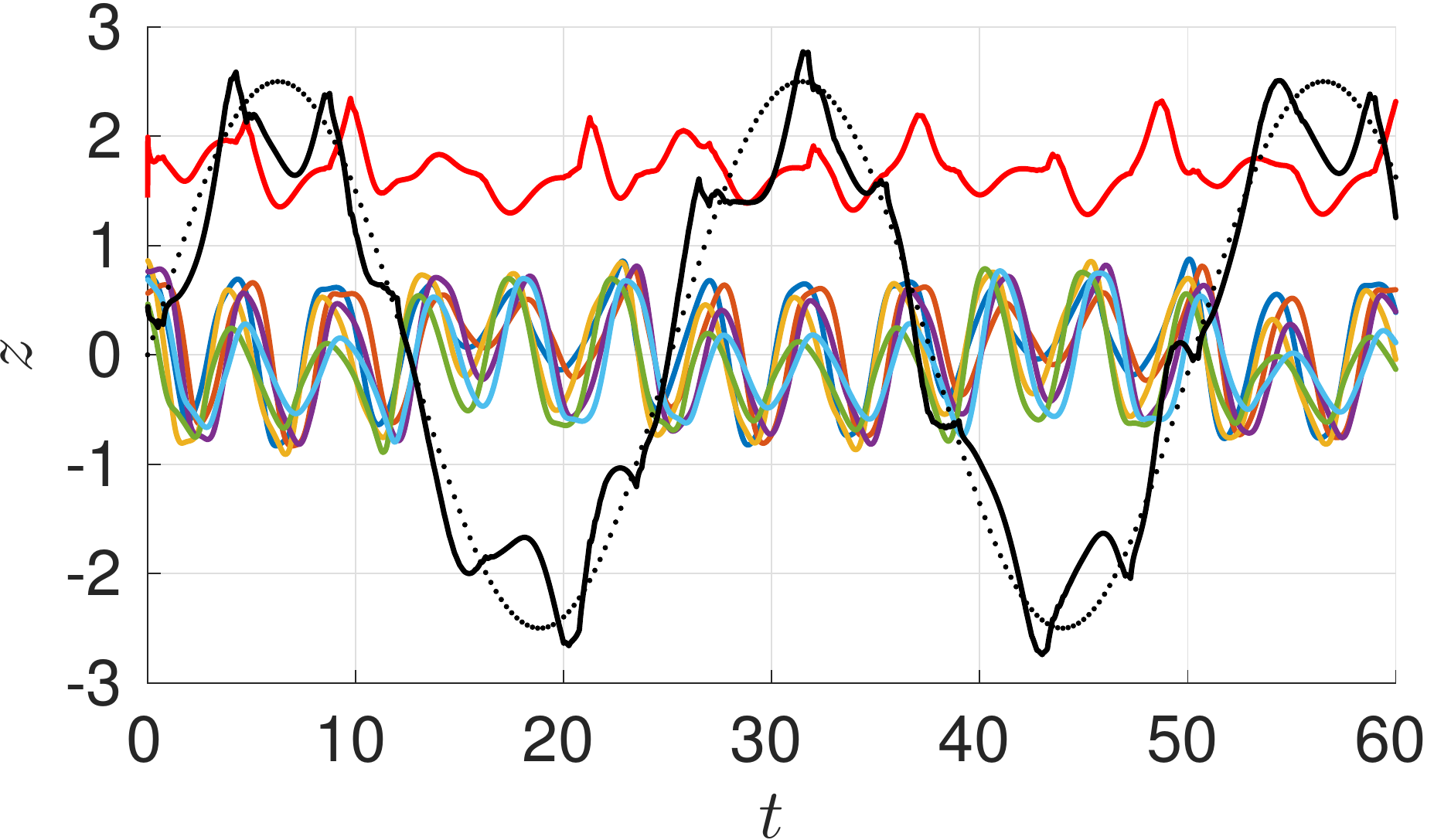}}
	\caption{(a) Snapshot of the PDE simulation. Colored dots: sensor positions $(\vec{x}_1,\ldots,\vec{x}_6)$. (b) Comparison between the observation of the PDE solution (dots) and the bilinear K-ROM \eqref{eq:KROM_continuous} ($u^0 = 0$, $u^1 = 2$) for $u(t) = 1 + sin(t)$. (c) Control trajectories obtained by \eqref{eq:MPC_STO_Koopman} (blue) and \eqref{eq:MPC_Bilinear_Koopman} (orange). (d) Objective function values corresponding to (c). (e) Trajectories of the observations $\vec{z}=f(\vec{y})$ corresponding to \eqref{eq:MPC_STO_Koopman}. Balck dots: reference trajectory for the lift coefficient (black line). Red line: drag coefficient. Horizontal velocities are colored according to the sensor positions in (a). (f) As in (e) but for \eqref{eq:MPC_Bilinear_Koopman}.}
	\label{fig:vonKarman}
\end{figure}

Similar to the Burgers example, we do not observe the full state. Since we want to control the vertical force on the cylinder (i.e., the lift), we directly observe the lift coefficient $C_l$. In addition, we observe the drag coefficient $C_d$ and the vertical velocity at six points $(\vec{x}_1,\ldots,\vec{x}_6)$ in the cylinder wake (see Figure~\ref{fig:vonKarman}~(a)), which yields the following observation:
\begin{equation*}
	\vec{z}(t) = \left(C_l(t), C_d(t), \vec{y}_2(\vec{x}_1,t), \ldots, \vec{y}_2(\vec{x}_6,t)\right)^{\top}.
\end{equation*}
As already mentioned, we want to influence the lift by rotating the cylinder. Since the lift coefficient is one of the observables, we simply have to track the corresponding entry of $\vec{z}$ in the MPC problem:
\begin{align*}
	\min_{\vec{\tau} \in \{u^0,u^1,u^2\}^p} \sum_{i=s}^{s+p-1} \left(\vec{z}_{i,1} - \vec{z}_{i}^{\mathsf{opt}}\right)^2.
\end{align*}

Here, we introduce three autonomous systems with the constant cylinder rotations $u_0 = -2$, $u_1 = 0$, $u_2 = 2$ for both K-ROM approaches.
The data is collected from one long-time simulation over $3000$ seconds with a random switching. As the lag time, we choose $h = 0.25$. 
The fact that we have three autonomous systems means that we use the localized ROM concept \eqref{eq:KROM_local} for the bilinear model. The MPC solutions to both reduced problem formulations (both with a prediction horizon of length $p=5$) are compared in Figure~\ref{fig:vonKarman}. We see in (b) a comparison between a PDE simulation and the bilinear K-ROM, and very good agreement is observed despite the significant dimension reduction. In Figure~\ref{fig:vonKarman} (c) and (d) the two K-ROM based solutions are compared and we observe almost equal quality of the solution. Interestingly, the switching approach is superior to the bilinear K-ROM. The reason for this is likely that due to the localized K-ROM approach, the objective function possesses many non-smooth kinks for $p=5$. Consequently, the true optimum is difficult to compute numerically without using algorithms specifically tailored to continuous, piece-wise smooth problems. Furthermore, inaccuracies introduced via linear interpolation become more significant for longer prediction horizons, i.e., for larger $p$.

\subsection{Data sampling and basis selection}
\label{subsec:data_sampling}

We have seen in the previous section that both approaches are capable of controlling complex PDE-constrained problems in real time. 
As already mentioned, the convergence result for EDMD only holds for infinitely large dictionaries $\Psi$ and infinitely many data points. Consequently, we now study the influence of different numerical parameters on the solution quality since these assumptions are not met in a practical setting.

\begin{figure}[b!]
	\centering
	\parbox[b]{0.49\textwidth}{\centering (a) \\ \includegraphics[width=.45\textwidth]{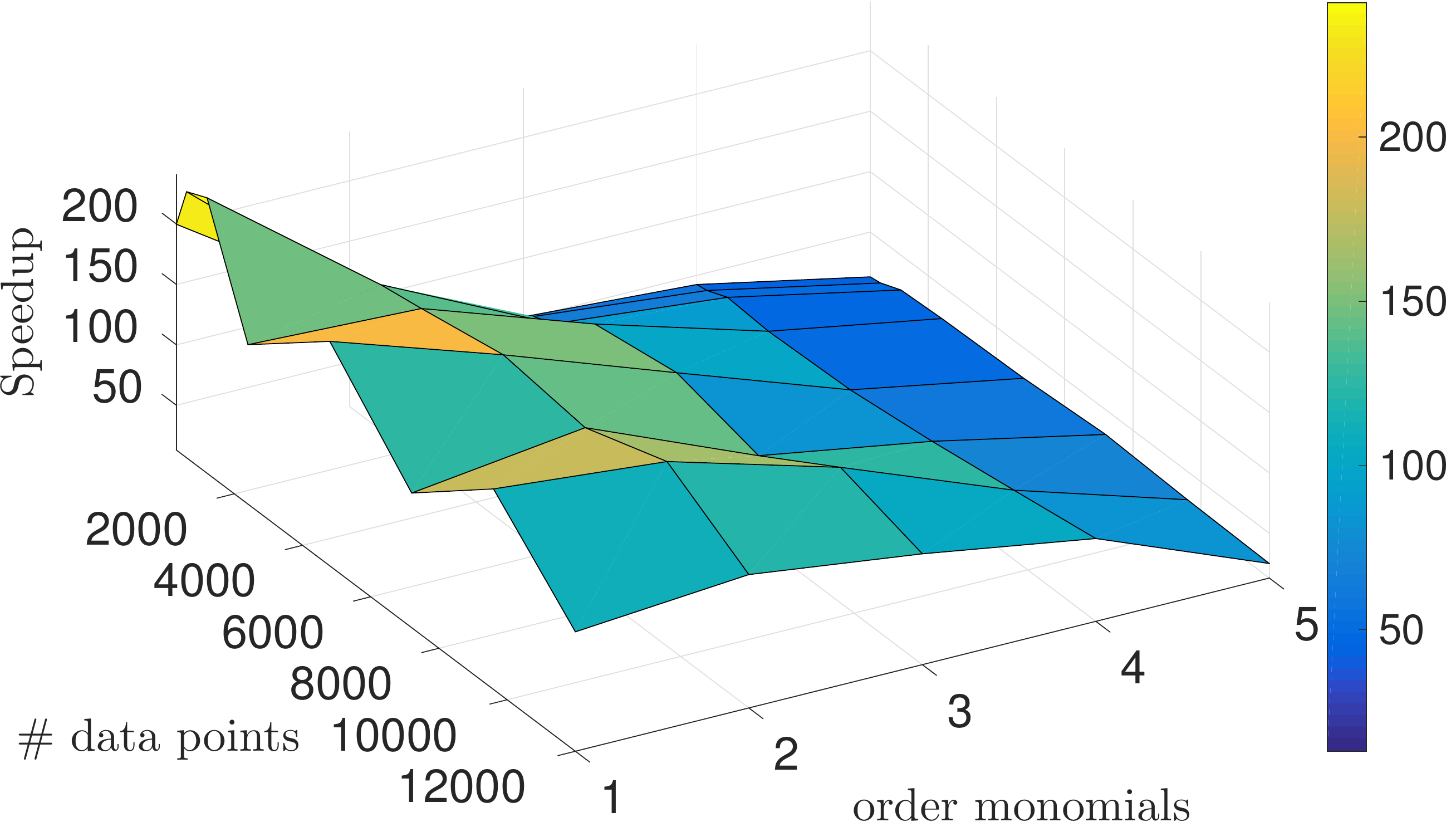}} \hfil
	\parbox[b]{0.49\textwidth}{\centering (b) \\ \includegraphics[width=.45\textwidth]{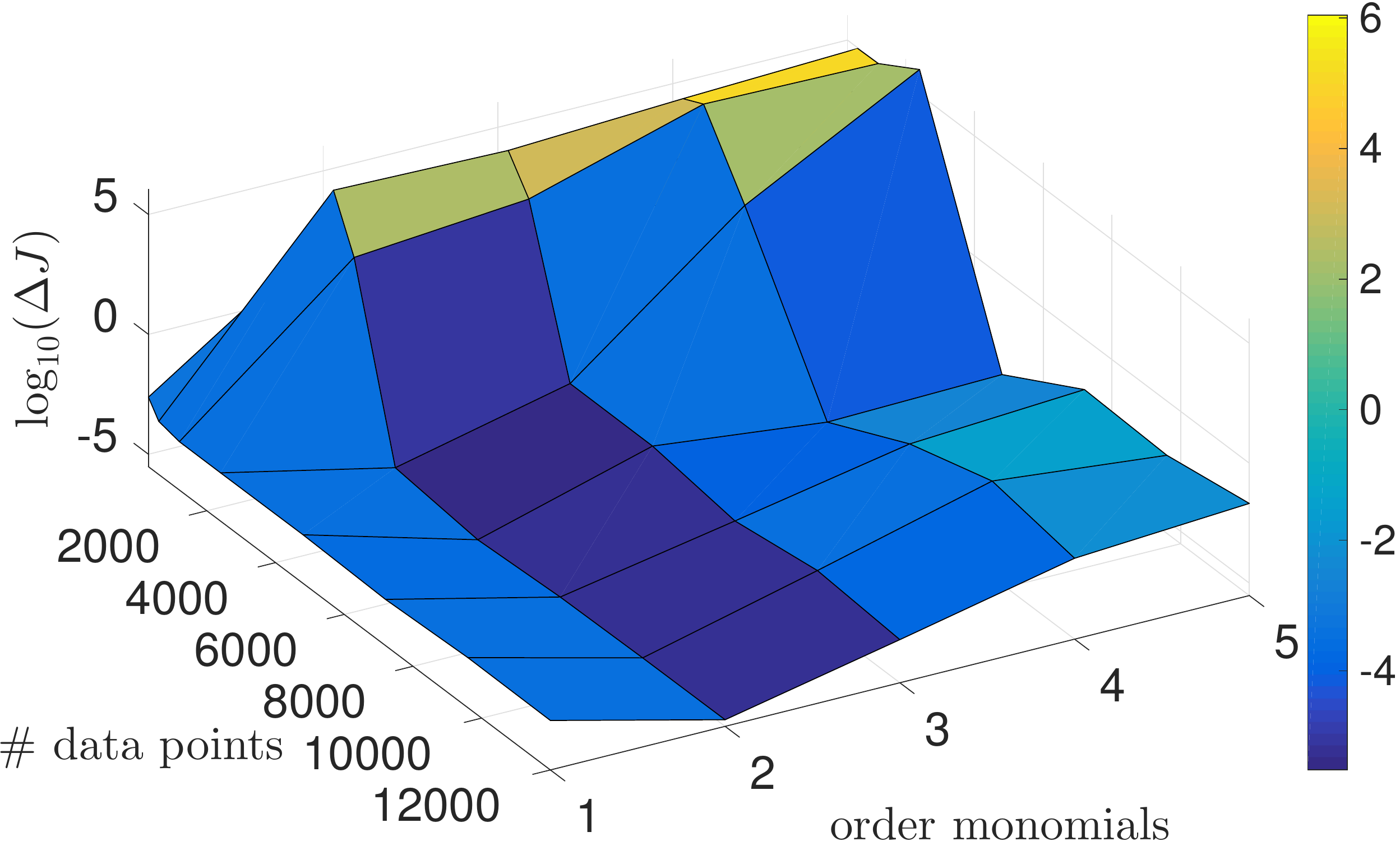}} \\[1ex]
	\parbox[b]{0.49\textwidth}{\centering (c) \\ \includegraphics[width=.45\textwidth]{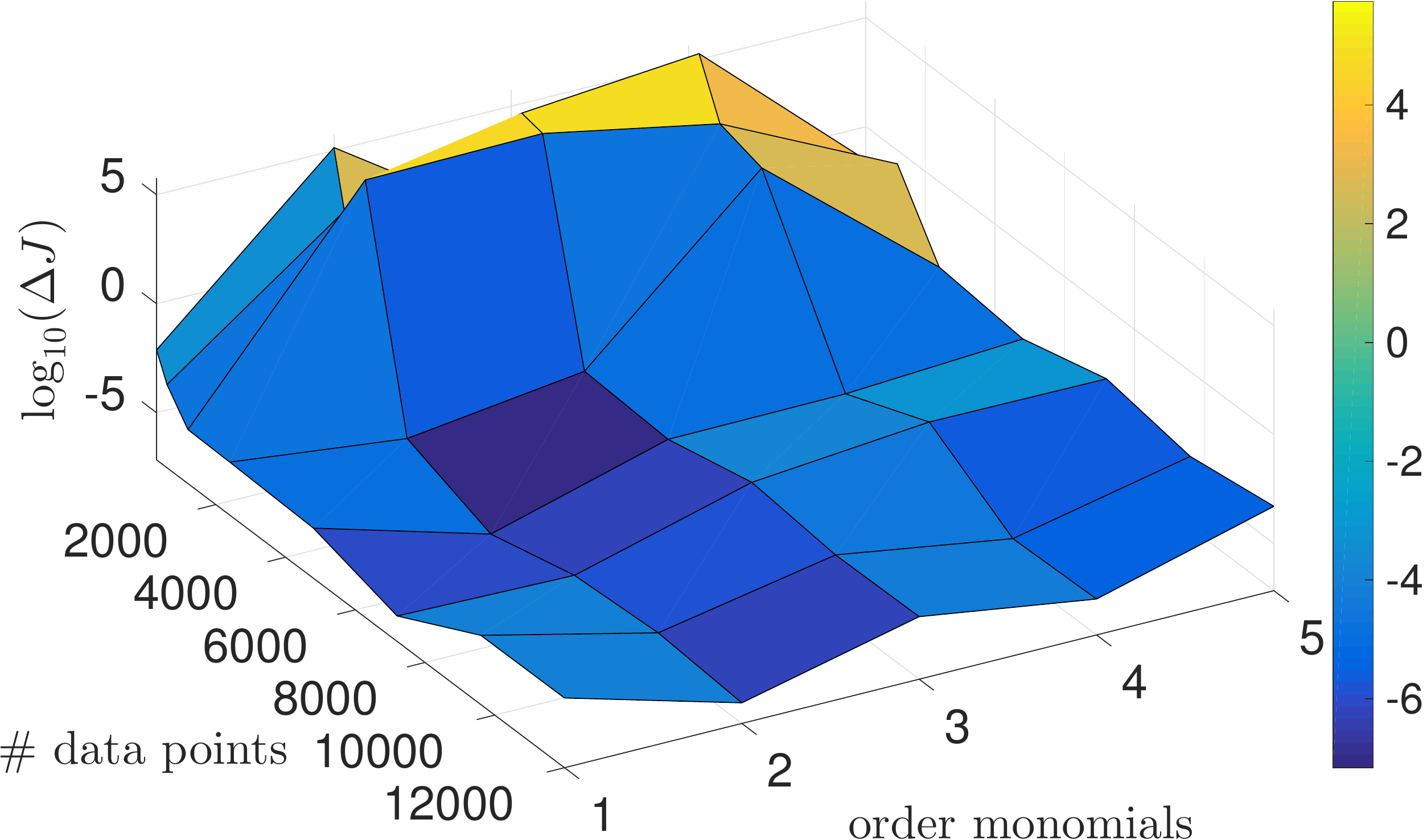}} \hfil
	\parbox[b]{0.49\textwidth}{\centering (d) \\ \includegraphics[width=.45\textwidth]{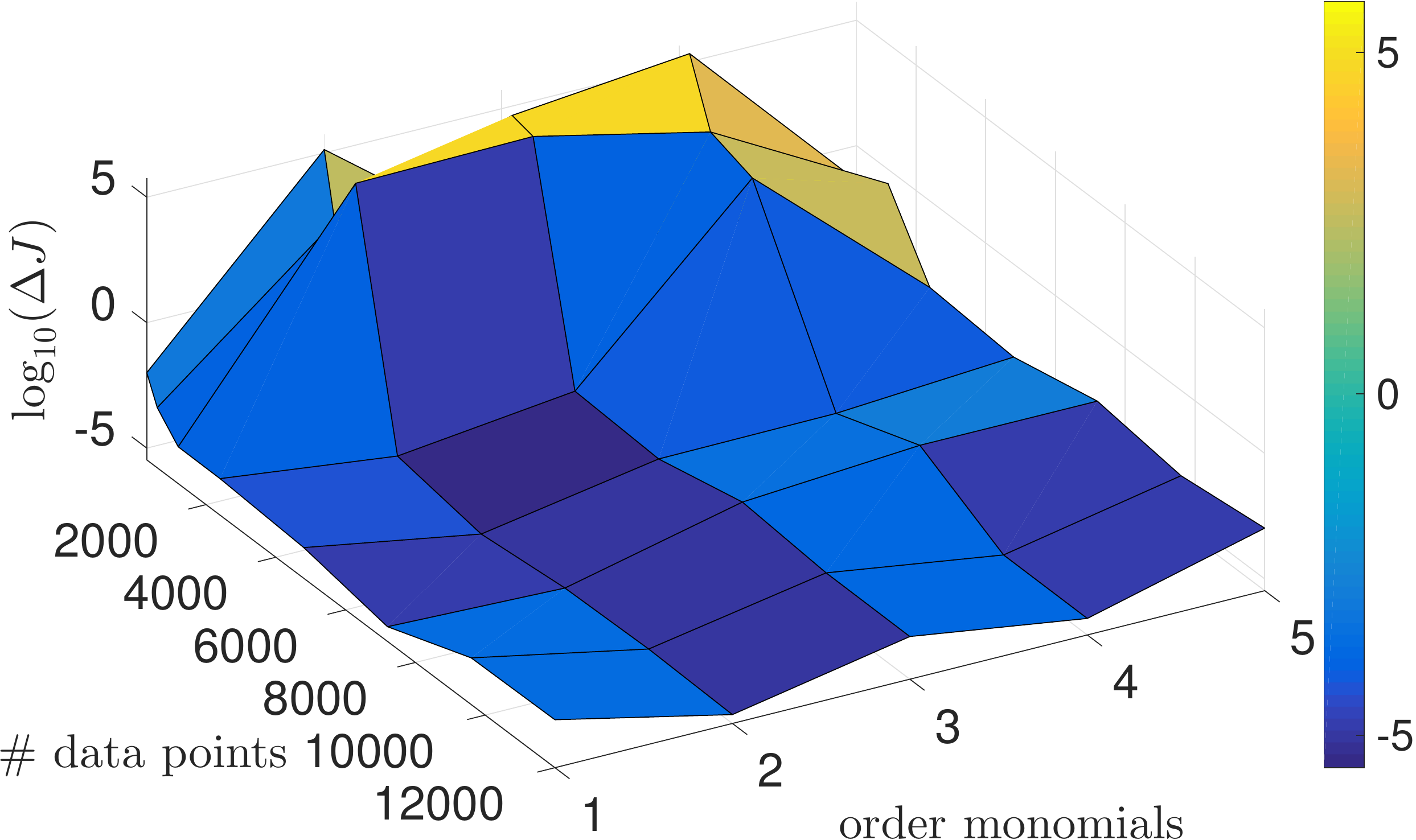}}
	\caption{Analysis of the performance of the two K-ROM approaches depending on the order of monomials as well as the number of data points used for the training for the 1D Burgers equation. (a) Speedup in comparison to the PDE solver. (b) to (d) Integrated error $\Delta J$ between the K-ROM and the full solution. (b) \eqref{eq:MPC_Bilinear_Koopman} vs.~\eqref{eq:MPC}. (c) \eqref{eq:MPC_STO_Koopman} vs.~\eqref{eq:MPC_STO}. (d) \eqref{eq:MPC_STO_Koopman} vs.~\eqref{eq:MPC}.}
	\label{fig:Analysis_Data_orderMon_Burgers}
\end{figure}

\begin{figure}[b!]
	\centering
	\parbox[b]{0.49\textwidth}{\centering (a) \\ \includegraphics[width=.45\textwidth]{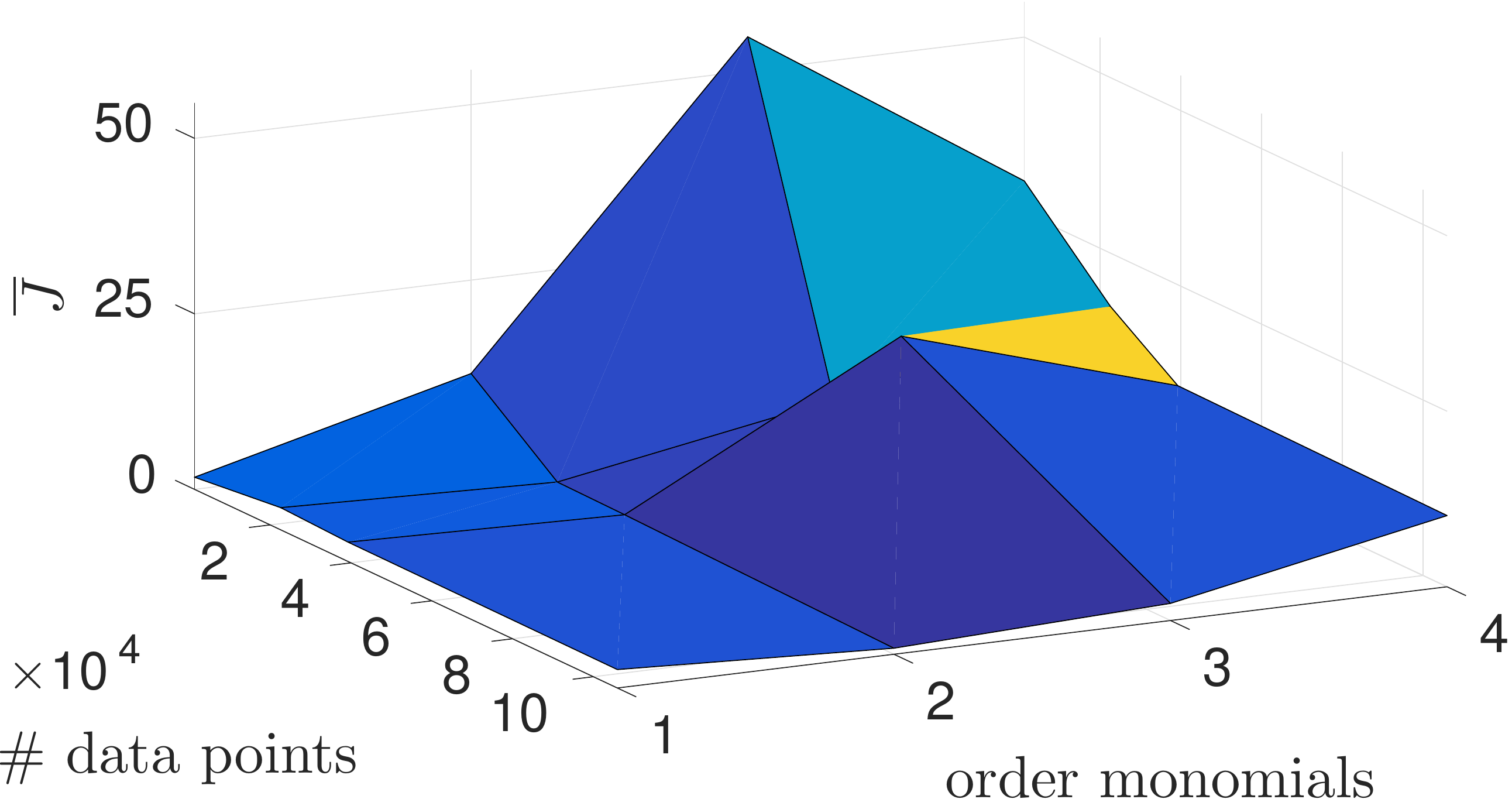}} \hfil
	\parbox[b]{0.49\textwidth}{\centering (b) \\ \includegraphics[width=.45\textwidth]{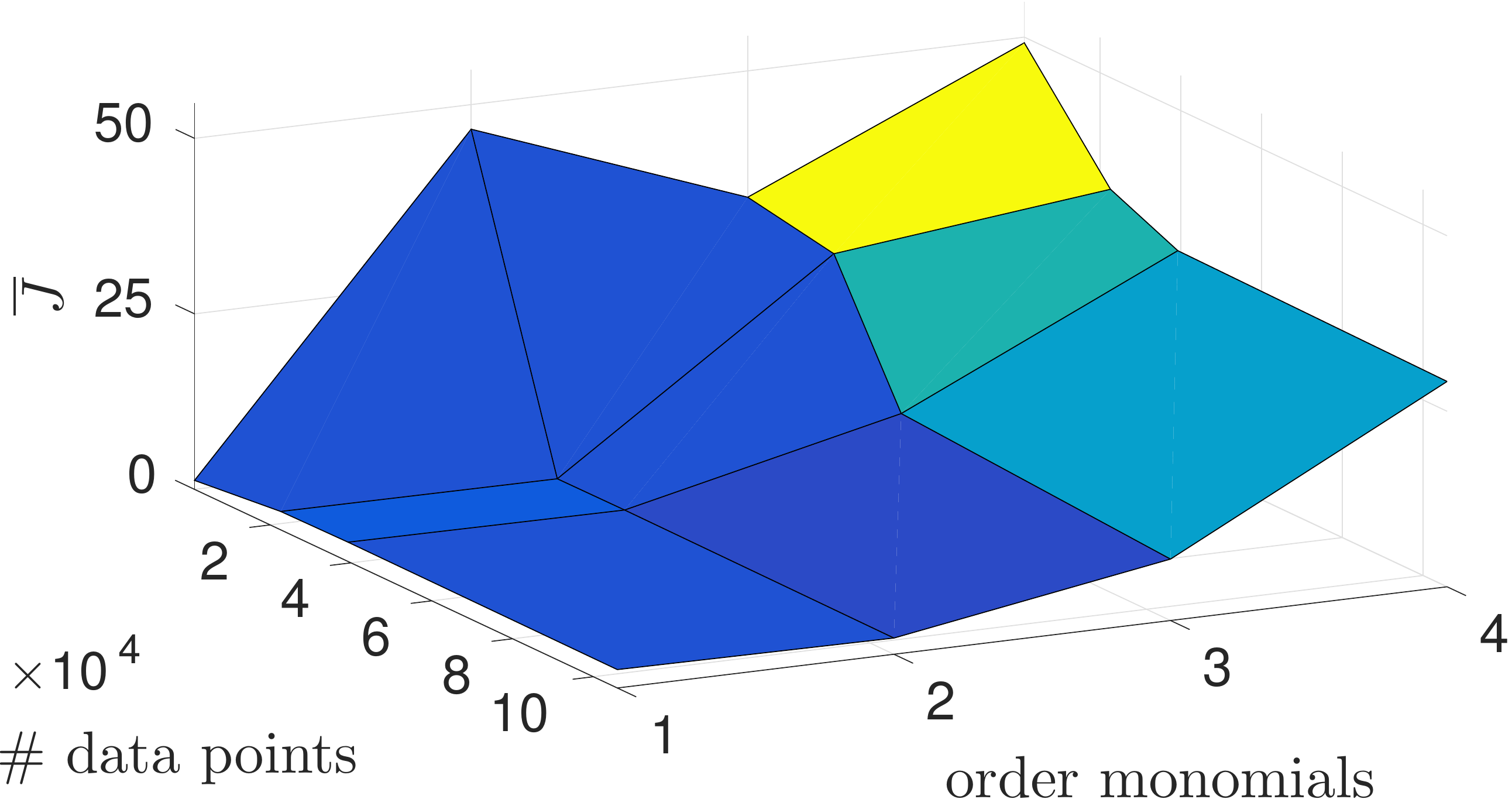}}
	\caption{Analysis of the performance of the two K-ROM approaches depending on the order of monomials as well as the number of data points used for the training for the 2D Navier--Stokes equations. (a) Integrated objective function value for \eqref{eq:MPC_STO_Koopman}. (b) Integrated objective function value for \eqref{eq:MPC_Bilinear_Koopman}.}
	\label{fig:Analysis_Data_orderMon_vonKarman}
\end{figure}

In Figure~\ref{fig:Analysis_Data_orderMon_Burgers}, the influence of the maximal order of the monomial basis and the size of the training data set is visualized for the Burgers example. In accordance with Section~\ref{subsubsec:Burgers}, we have taken three K-ROMs for the switched systems approach and two for the bilinear model. The order of the polynomials directly influences the dimension of the K-ROM such that the speedup crucially depends on this choice, see Figure~\ref{fig:Analysis_Data_orderMon_Burgers}~(a), where we have speedup factors of approximately 200 for a maximum order of 1 (i.e., standard DMD). This factor then reduces to roughly 50 for polynomials up to order 5. However, we see that for both K-ROM approaches, it is sufficient to consider monomials of order 2. Another benefit of these lower-dimensional K-ROMs is that the amount of required data is smaller. The larger the surrogate model is, the more data we need to compute satisfactory approximations of the Koopman operator. In fact, it appears that the standard DMD approach is the most robust concerning the amount of training data. When comparing Figure~\ref{fig:Analysis_Data_orderMon_Burgers}~(b) and (d), where the distance between the K-ROM approaches and the continuous PDE constrained MPC problem are compared, we see that -- surprisingly -- the switched systems approach yields a feedback behavior of similar quality compared to the bilinear K-ROM. Note, however, that the amount of training data is smaller for the bilinear model.  

When studying the Navier--Stokes example, the picture is very similar, cf.~Figure~\ref{fig:Analysis_Data_orderMon_vonKarman}. Since the PDE constrained solution is not available, we here plot the integrated objective function value of the PDE model: $\overline{J} = \int_{t_0}^{t_e} J(t)\, dt$. Similar to the previous example, the DMD based version appears to be the most robust concerning the data requirements. Again, both approaches are very similar in performance, and considering monomials of order larger than two is significantly inferior.

\begin{remark}[Influence of the number of K-ROMs]
	As a final experiment, we study how the number of localized K-ROMs influences the solution. To this end, we revisit the Burgers example with three different discretizations of $u$. We consider the cases with two, three, and five inputs at which data is available.
	We see in Figure~\ref{fig:Analysis_numberKROMs_Burgers} that no real trend can be identified. It is interesting to see, however, that no improvement can be observed when increasing the number of control inputs from three to five. On the contrary, increasing the number of ROMs may even be disadvantageous, in particular if the amount of data is insufficient for a good approximation of the Koopman operator. In conclusion, a small to moderate number of control inputs in combination with the classical DMD approximation appears to be the most efficient and robust choice for the problems studied here.
	\begin{figure}[h!]
		\centering
		\parbox[b]{0.49\textwidth}{\centering (a) \\ \includegraphics[width=.45\textwidth]{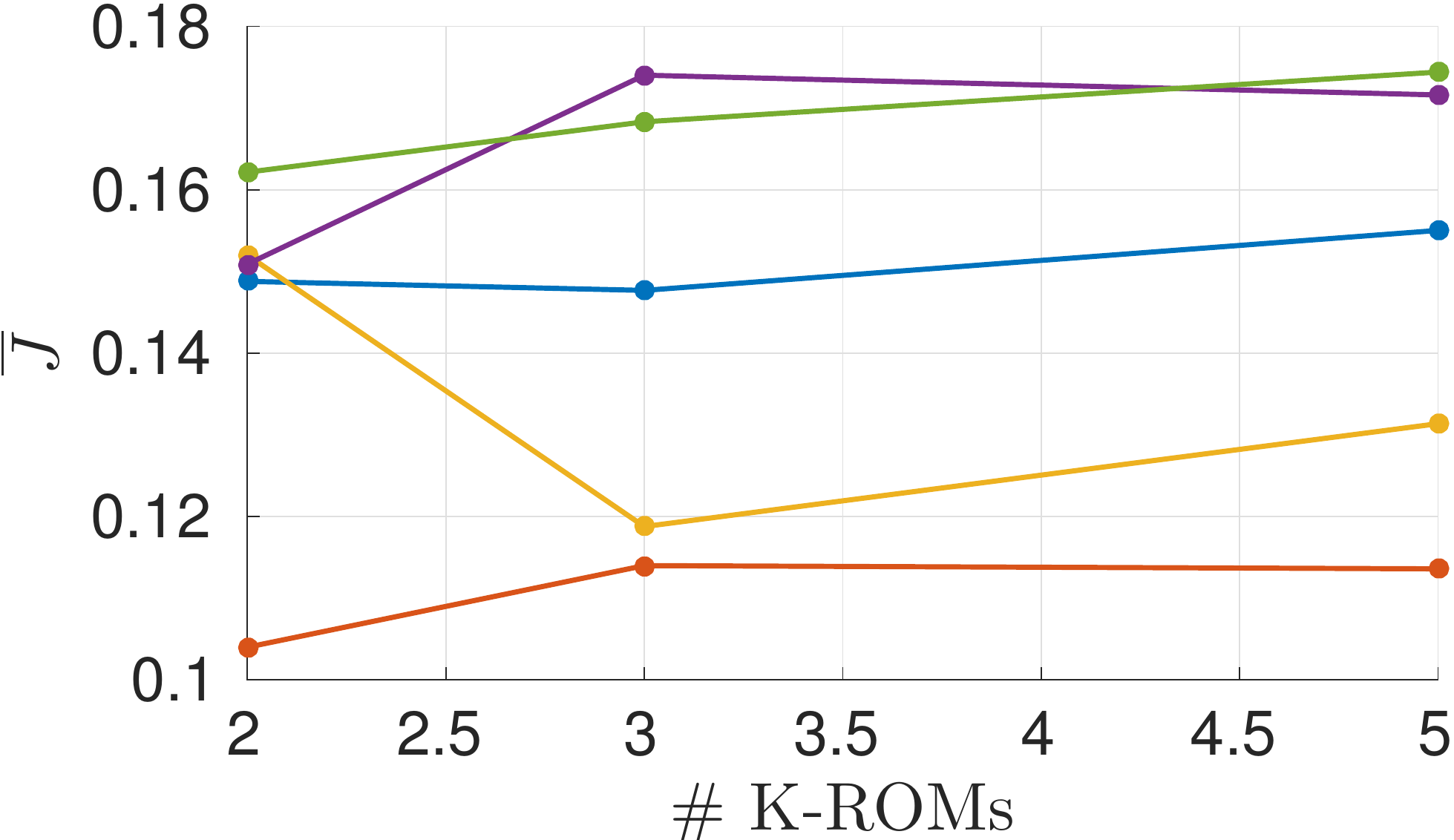}} \hfil
		\parbox[b]{0.49\textwidth}{\centering (b) \\ \includegraphics[width=.45\textwidth]{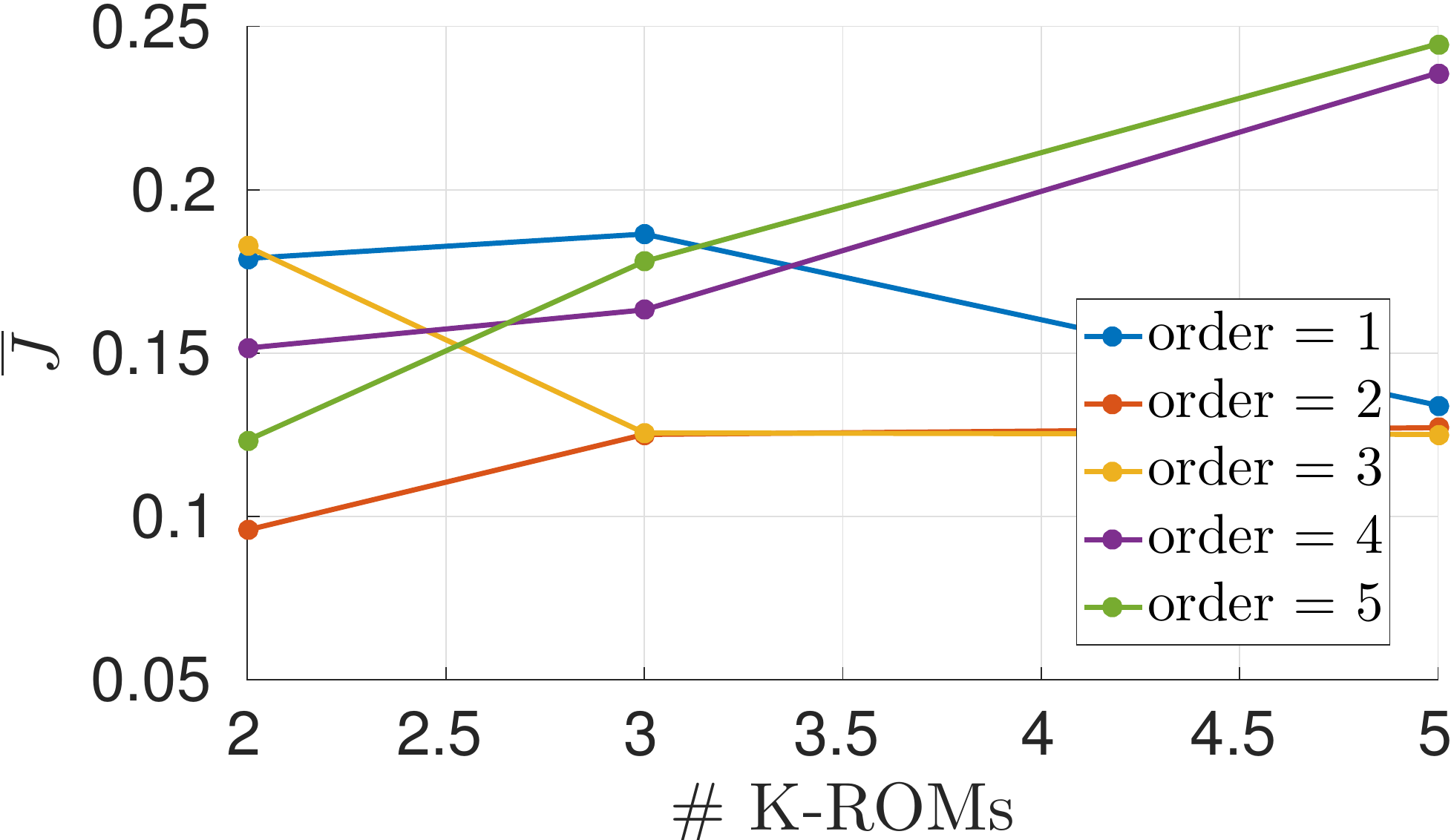}}
		\caption{Integrated objective function value for \eqref{eq:MPC_STO_Koopman} (a) and \eqref{eq:MPC_Bilinear_Koopman} (b), evaluated at a medium data value (i.e., \# data points $\approx$ 4000, cf.~Figure~\ref{fig:Analysis_Data_orderMon_Burgers}~(d)).}
		\label{fig:Analysis_numberKROMs_Burgers}
	\end{figure}
	\exampleSymbol
\end{remark}

\section{Online updates using sensor data}
\label{sec:OnlineUpdates}

During the MPC algorithm, new sensor data is obtained at every sample time $h$. Furthermore, these data points are collected in regions of state space which are of particular interest since they are close to the desired state. Therefore, it is a natural idea to use this data to further improve the quality of the K-ROM. To this end, we make use of the idea developed in \cite{HWR14}, where incremental updates of the DMD approximation are performed in a way that we do not have to store data points that have already been taken into account. We make use of the transformation
\begin{align*}
	\vec{U}^\top = \vec{A} \vec{G}^{+},
\end{align*}
with 
\begin{align*}
	\vec{A} &= \frac{1}{m} \sum_{i=1}^{m} \Psi(\widetilde{\vec{z}}_i) \Psi(\vec{z}_i)^\top,\\
	\vec{G} &= \frac{1}{m} \sum_{i=1}^{m} \Psi(\vec{z}_i) \Psi(\vec{z}_i)^\top,
\end{align*}
see \cite{WKR15,KKS16} for details. We see that in order to update $\vec{U}^\top$, we merely have to store the (in our case low-dimensional) matrices $\vec{A}$ and $\vec{G}$. Each time we obtain a new snapshot pair $(\vec{z}_{m+1}, \widetilde{\vec{z}}_{m+1})$ we can update the EDMD approximation via
\begin{align*}
	\widehat{\vec{A}} &= \frac{m  \vec{A} + q \left(\Psi(\widetilde{\vec{z}}_{m+1}) \Psi(\vec{z}_{m+1})^\top\right)}{m+q},\\
	\widehat{\vec{G}} &= \frac{m  \vec{G} + q \left(\Psi(\vec{z}_{m+1}) \Psi(\vec{z}_{m+1})^\top\right)}{m+q}, \\
	\widehat{\vec{U}}^\top &= \widehat{\vec{A}} \widehat{\vec{G}}^{+},
\end{align*}
where $q \in \N^{\geq 1}$ is a weight parameter. Note that we have to set $q=1$ in order to obtain the standard EDMD procedure. Alternatively, we can determine $q$ in such a way that the update has a higher impact, e.g., by some prescribed percentage $\epsilon$:
\begin{align*}
	q =\left \lfloor m \frac{\epsilon}{1 - \epsilon} \right \rfloor.
\end{align*}
The most expensive part for this update is to compute the pseudo-inverse of $\widehat{\vec{G}}$. However, since the K-ROM dimension is generally low, it can be computed efficiently.

\begin{figure}[t!]
	\centering
	\parbox[b]{0.8\textwidth}{\centering (a) \\ \includegraphics[width=.8\textwidth]{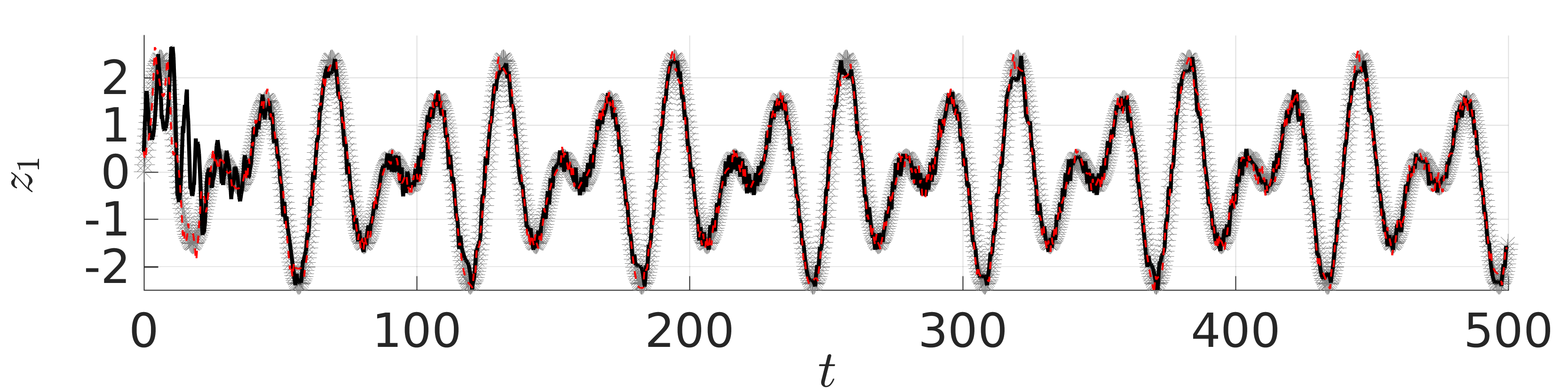}}\\[1ex]
	\parbox[b]{0.8\textwidth}{\centering (b) \\ \includegraphics[width=.8\textwidth]{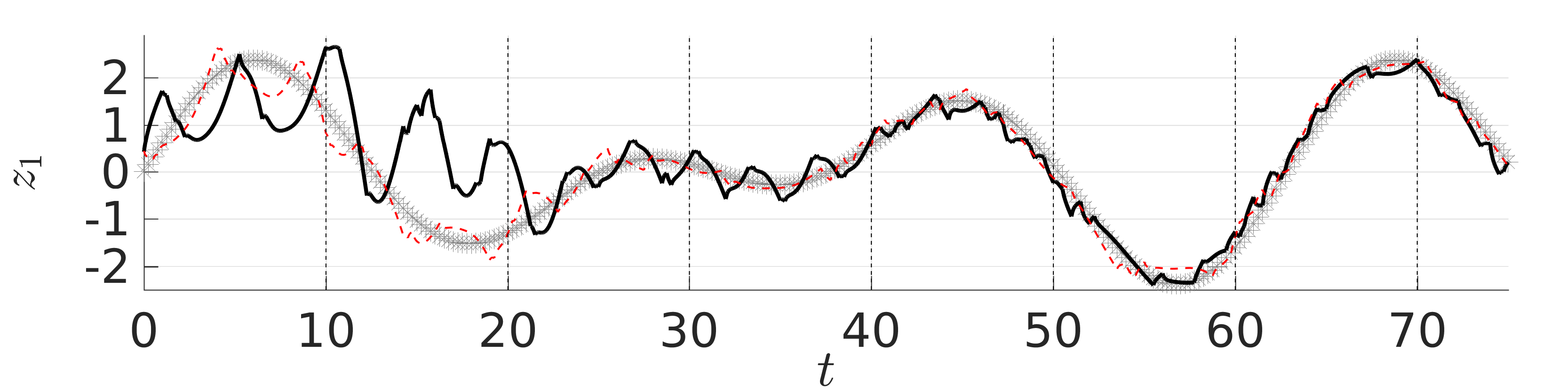}}\\[1ex]
	\parbox[b]{0.8\textwidth}{\centering (c) \\ \includegraphics[width=.8\textwidth]{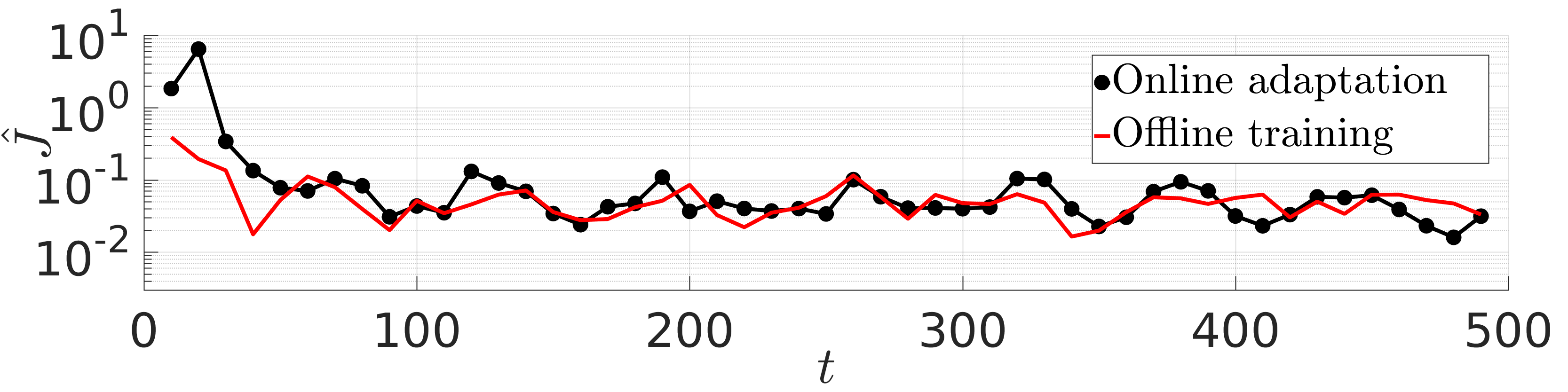}} 
	\caption{Online K-ROM adaptation within \eqref{eq:MPC_STO_Koopman}. (a) Optimal lift trajectory obtained via online adaptation (black) and the standard approach with offline training (red). The reference trajectory is marked by the grey stars. (b) As in (a), zoom into the first 75 seconds. The vertical dashed lines indicate the time instances at which the K-ROMs are updated. (c) Objective function value integrated over the past 10 seconds. The dots mark the points where the K-ROMs are updated.}
	\label{fig:NSE_Learning}
\end{figure}

We now apply this procedure to the Navier--Stokes example. As we have seen in Section~\ref{subsec:data_sampling}, the standard DMD approximation is most stable in the low data limit. To illustrate the data efficiency of our approach, we consequently choose DMD for the K-ROM computation in this section and start with only 50 data points for each of the three autonomous systems $u^0 = -2$, $u^1 = 0$ and $u^2 = 2$.

We set $\epsilon = 0.025$ and 
store the matrices $\vec{A}$ and $\vec{G}$ and sensor data that we measure during the MPC algorithm. We then update the three autonomous systems every 10 seconds. Note that this approach is only viable for the switching approach \eqref{eq:MPC_STO_Koopman} since otherwise, we collect data at intermediate control values for which we do not have a reduced order model. For this approach, the setup proposed in \cite{KM16,AKM18} would be more appropriate.

Figure~\ref{fig:NSE_Learning} shows the results for the lift tracking problem, where we compare the online adaptation with the pure offline training considered before. For the comparison, we consider a longer trajectory of $500$ seconds with 
\[
	z^{\mathsf{opt}}(t) = 2.5  \sin\left(\frac{t}{4}\right) \cos\left(\frac{t}{20}\right).
	\] 
In order to evaluate the online procedure, we compare the objective function value integrated over the past 10 seconds, i.e., over the period with the current K-ROM approximation:
\[
	\hat{J}(t) = \int_{t - 10}^{t} \left(\vec{z}_1(\tau) - z^{\mathsf{opt}}(\tau)\right)^2\, d\tau. 
\]
We see that after the first four to five updates, we already have a very good agreement. Consequently, a K-ROM based controller can be set up with a very small amount of data and then be updated regularly. After a short period of time, the performance is equal to that with an extensive offline phase.

It should be noted that this approach is only applicable if the system cannot stray arbitrarily far from the reference trajectory. Otherwise, it could happen that a strong deterioration occurs from the beginning. Consequently, the resulting training data would not be suited for improving the control performance.

\section{Conclusion}
\label{sec:Conclusion}

We have presented two methods for solving PDE constrained optimal control problems using the Koopman operator for significant speedup. In order to increase the data efficiency, we only collect data for a small number of constant inputs. The corresponding control problem then becomes a switching problem. Alternatively, we obtain a bilinear surrogate model via linear interpolation between two K-ROMs. Based on a recent convergence result for EDMD, convergence of the K-ROM based optimization problems can be shown. Extensive numerical studies show the applicability of both methods to nonlinear PDE constrained problems. Interestingly, a simple DMD approximation of the Koopman operator is beneficial both for the solution quality and robustness with respect to the amount of training data. The reason for the good performance is that predictability only has to be accurate for a small number of time steps. Finally, an extension to online adaptations using sensor data has been presented and validated. This way, we can set up real-time controllers with very limited data.

For future work, it will be interesting to validate the methodology in experiments. First results for ODE constrained problems (electrical drives) are promising~\cite{HPW+18}. Furthermore, the system dynamics of the examples considered here are fairly well-behaved. Consequently, it will be of great interest to study the presented approaches for more complex dynamical systems, e.g., for decreasing viscosity / increasing Reynolds numbers.

\bibliographystyle{unsrt}
\bibliography{Bibliography}
\end{document}